\documentclass[10pt, reqno]{amsart}
\usepackage{finite_subgroups}

\title{Finite $p$-Irregular Subgroups of $\PGL_2(k)$}
\author{Xander Faber}
\address{IDA Center for Computing Sciences \\
Bowie, MD}
\email{awfaber@super.org}

\thanks{I originally wrote this article in 2011 and then let it languish for a
  decade because of a lackluster response from several journals. I am grateful
  to several mathematicians who have reached out in the intervening years to
  encourage another attempt at publication. In addition, I would like to thank
  Geoff Robinson for pointing me to Dickson's work on the subject, and Mike
  Zieve for suggesting the papers of Gierster and E.H. Moore. I held a National
  Science Foundation Postdoctoral Fellowship from 2009-2010 when I started this
  project.}

\begin{document}
\begin{abstract}
In the late 19th century, Klein inaugurated a program for describing the finite
subgroups of $\PGL_2(k)$ by treating the case in which the field $k$ is the
complex numbers. Gierster and Moore extended Klein's arguments to deal with
finite fields. In the past century, additional contributions to this problem
were made by Serre, Suzuki, and Beauville, among others.  We complete this
program by giving a classification of the finite subgroups of $\PGL_2(k)$ with
order divisible by $p$, up to conjugation, for an arbitrary field $k$ of
positive characteristic~$p$.
\end{abstract}

\maketitle


\section{Introduction}
Any finite group of rotations of the round 2-sphere describes a regular polygon
or polytope: look at the action of the group on a point of the sphere. In the
1870s, Klein showed that the symmetry groups of regular polygons and polytopes
are the \textit{only} way to obtain finite groups of rotations of the sphere
\cite[Ch.V.2]{Klein_Icosahedron}. Klein's approach was to use stereographic
projection to identify the 2-sphere in $\RR^3$ with the Riemann sphere $\CC \cup
\{\infty\}$, and then to classify the finite subgroups of M\"obius
transformations $z \mapsto \frac{az + b}{cz+d}$. In modern parlance, we
could say that Klein classified the finite subgroups of $\PGL_2(\CC)$, the group
of invertible $2 \times 2$ complex matrices modulo scalar matrices.

One can also look at M\"obius transformations with coefficients in a finite
field, as Galois did in the early 1830s while studying symmetries of equations
\cite[pp.28-30]{Galois_oevres}. Later, Klein's technique was extended by his
student, Gierster, to give a complete description of the subgroups of
$\PGL_2(\FF_p)$ when $p$ is an odd prime \cite{Gierster_PSLp}. A decade after
that, E.H. Moore further extended this method to classify the subgroups of
$\PGL_2(\FF_q)$ when $q$ is an arbitrary prime power \cite{Moore_PSLq}. The
definitive version of this method appears in the writing of Moore's student,
L.E. Dickson \cite[\S239-261]{Dickson_Linear_Groups_1901}.  (To be precise, all
  of the authors in this paragraph considered subgroups of $\PSL_2(\FF_q)$, the
  subgroup of $\PGL_2(\FF_q)$ consisting of classes of matrices with square
  determinant. As $\PSL_2(\FF_q)$ is of index at most $2$ inside
  $\PGL_2(\FF_q)$, little is lost in discussing $\PGL_2(\FF_q)$ instead.)
	
More recently, Beauville gave a beautiful exposition of the
classification of finite subgroups of $\PGL_2(k)$ of order prime to
$\mathrm{char}(k)$ --- for an arbitrary field $k$ --- using Galois cohomology
\cite{Beauville_Finite_Subgroups_2010}. His proof takes advantage of the
accidental isomorphism $\PGL_2(k) \cong \mathrm{SO}(k, q)$, where the latter is
the special orthogonal group for the quadratic form $q(x,y,z) = x^2 + yz$. (This
isomorphism exists because $\PGL_2(k)$ is the automorphism group of the
projective line, which may be embedded in $\PP^2$ as the conic $x^2 + yz = 0$.)
Serre also gives an excellent discussion of existence results for certain
subgroups of $\PGL_2(k)$ in \cite[\S2]{Serre_proprietes_galoisiennes}.

In the present work, we extend these classifications to \textit{arbitrary}
finite subgroups $G \subset \PGL_2(k)$. The case of new interest is then an
infinite field of characteristic $p > 0$ and a subgroup $G$ whose order is
divisible by $p$ --- a $p$-irregular subgroup. We learn that a finite
$p$-irregular subgroup of $\PGL_2(k)$ is isomorphic to $\PSL_2(\FF_q)$ or
$\PGL_2(\FF_q)$ for $q$ a power of the characteristic of $k$, to a
$p$-semi-elementary group, to a dihedral group, or to the alternating group on 5
letters. Theorems~A and~B in the next section are our novel contributions in
this article: Theorem~A gives precise conditions for the existence of
$p$-irregular subgroups, and Theorem~B gives the classification of such
subgroups up to conjugacy inside $\PGL_2(k)$.
	
The subgroups of $\PGL_2$ came to the author's attention as automorphism groups
of rational functions, viewed as discrete dynamical systems on the projective
line. Given a rational function $\phi \in \QQ(z)$, the automorphism group is
$\Aut_\phi(\QQ) = \{f \in \PGL_2(\QQ) : f \circ \phi \circ f^{-1} = \phi\}$. The
author, Manes, and Viray designed an algorithm for computing $\Aut_\phi(\QQ) $
by piecing it together from $\Aut_\phi(\FF_p) \subset \PGL_2(\FF_p)$ for several
primes $p$ of good reduction for $\phi$ \cite{FMV}; evidently the classification
of subgroups of $\PGL_2(\FF_p)$ up to conjugation is invaluable for such a
pursuit. Such a classification does not seem to appear explicitly in the
literature in a form that can be readily cited, and so we deduce it as a special
case of the main results of the present work. See Theorem~D in the next
section.
	
This article consists of two main parts, which are somewhat disparate in
nature. The first part --- Sections~\ref{Sec: Fixed Points} through~\ref{Sec:
  p-irregular} --- is devoted to generalizing the arguments of Klein, Gierster,
and Moore, following Dickson's exposition
\cite[\S239--261]{Dickson_Linear_Groups_1901}, in order to give a complete
classification of finite subgroups of $\PGL_2(k)$ when $k$ is algebraically
closed.\footnote{A list of isomorphism types of subgroups of $\PSL_2(k)$ for $k$
algebraically closed appears in Suzuki's book
\cite[p.404]{Suzuki_Group_Theory_I}, but no description of conjugacy classes is
given. In addition, Suzuki incorrectly attributes the classification of
subgroups of $\PSL_2(\FF_q)$ to Dickson \cite[p.392]{Suzuki_Group_Theory_I}.}
The argument is elementary --- it uses only some basics of group actions, the
Sylow theorems, and a little knowledge of the action of $\PGL_2(k)$ on
$\PP^1(k)$. We have endeavored to keep this part of the article completely
self-contained; it could serve as the basis for a reading project by an
undergraduate with a background in abstract algebra. As a natural byproduct of
this approach, we were able to recover the classical description of finite
$p$-regular subgroups of $\PGL_2(k)$ for an algebraically closed field $k$
(cf. Theorem~C). We stress that nothing in this part of the article is new;
however, the next part depends heavily on these statements.
	
The second part of this paper --- Sections~\ref{sec:arithmetic_p_regular}
and~\ref{Sec: Proofs} --- completes the classification of finite subgroups of
$\PGL_2(k)$ by first passing to separably closed fields, and then by applying
Galois descent. This part is less elementary and draws heavily from Beauville's
paper \cite{Beauville_Finite_Subgroups_2010}. The main idea is to use a
cohomological parameterization of the set of conjugacy classes of subgroups in
$\PGL_2(k)$ that coincide over a separable closure of $k$. For $p$-irregular
subgroups, the relevant cohomology set turns out to be trivial, and so we reduce
to the case of separably closed fields. In this sense, the classification of
$p$-irregular subgroups is \textit{easier} than its $p$-regular counterpart.
	
We close this section with a more precise description of the contents of the
article. Section~\ref{sec:theorems} contains our notational conventions and the statements
of the main theorems. Section~\ref{Sec: Fixed Points} collects a number of
invaluable trace/determinant equations that characterize when an element of
$\PGL_2(k)$ has a certain small order. In Section~\ref{Sec: Special Subgroups},
we study several seemingly special subgroups of $\PGL_2(k)$; we then show that
these special cases exhaust all possible finite subgroups in Sections~\ref{Sec: p-regular} ($p$-regular subgroups) and~\ref{Sec: p-irregular} ($p$-irregular
subgroups). We will restrict to the case in which $k$ is algebraically closed in
Sections~\ref{Sec: Fixed Points} through~\ref{Sec: p-irregular}. Ignoring the
question of when certain equations actually have solutions in $k$ clarifies the
presentation dramatically. In the remainder of the paper, we will want to lean on
the results of Serre and Beauville on existence and classification of
$p$-regular subgroups of $\PGL_2(k)$; we recall these in the form we require in
Section~\ref{sec:arithmetic_p_regular}.  In Section~\ref{Sec: Separably closed}
we pass to separably closed fields $k$; this requires extra work only in
characteristic~2. Section~\ref{Sec: Descent} contains the Galois cohomology
computation necessary to pass from separably closed fields to the general
case. We finish in Section~\ref{Sec: Proofs} with proofs of the main theorems.


\section{The Classification}
\label{sec:theorems}

Let $k$ be a field. Define $\GL_2(k)$ to be the group of invertible $2 \times 2$
matrices with entries in $k$; it contains an isomorphic copy of $k^\times$ given
by multiples of the identity matrix. We write $\PGL_2(k) = \mathrm{GL}_2(k) /
k^\times$, and elements of $\PGL_2(k)$ will be represented by equivalence
classes of matrices $\mat{\alpha & \beta \\ \gamma & \delta}$ with $\alpha
\delta - \beta \gamma \neq 0$. Here $\mat{\alpha & \beta \\ \gamma & \delta} =
\mat{\lambda \alpha & \lambda \beta \\ \lambda \gamma & \lambda \delta}$ for any
$\lambda \neq 0$. We write $I$ for the identity of $\PGL_2(k)$. An element $s
\in \PGL_2(k)$ acts on $\PP^1(k) = k \cup \{\infty\}$ by the formula $s.z =
\frac{\alpha z + \beta}{\gamma z + \delta}$, with the usual conventions that
$s.\infty = \alpha / \gamma$ and $\alpha / 0 = \infty$.

The determinant map $\det: \mathrm{GL}_2(k) \to k^\times$ descends to a
homomorphism $\overline{\det}: \PGL_2(k) \to k^\times / (k^\times)^2$. Write
$\PSL_2(k)$ for the kernel of $\overline{\det}$.

The letter $q$ will always denote a power of $p$. We say that an element of a
finite group $G$ is $p$-regular if its order is prime to $p$, and we say that
$G$ is \textbf{$p$-regular} if all of its elements are $p$-regular. Otherwise,
an element or a group is \textbf{$p$-irregular}.

For the following statements, we will use the notation $\mathfrak{D}_{n}$,
$\mathfrak{S}_n$, and $\mathfrak{A}_n$ for the dihedral group with $2n$
elements, the symmetric group on $n$ letters, and the alternating group on $n$
letters, respectively. Write $\mu_n(k)$ for the group of $n$-th roots of unity
in $k$.

An abstract finite group $G$ will be called \textbf{$p$-semi-elementary} if it
has a unique Sylow $p$-subgroup $P$ of exponent $p$ with $G / P$ cyclic.  If $G
/ P$ is trivial, we say that $G$ is \textbf{$p$-elementary}. Write $k_\alg$ for
an algebraic closure of $k$. A $p$-semi-elementary subgroup $G \subset
\PGL_2(k)$ acts on $\PP^1(k_\alg)$ with a unique fixed point (\S\ref{Sec:
  p-groups}). If the fixed point is $k$-rational, we say that $G$ is
\textbf{split}; otherwise, $G$ is non-split.

As we will see in Theorem~B, any finite $p$-irregular subgroup of $\PGL_2(k)$ is
isomorphic to $\PSL_2(\FF_q)$ or $\PGL_2(\FF_q)$ for some $q$, to a
$p$-semi-elementary group, to a dihedral group, or to $\mathfrak{A}_5$. The
following theorem describes the precise conditions under which these types of
subgroups exist in $\PGL_2(k)$.

\begin{thmA*}[Existence of Finite $p$-Irregular Subgroups]
Let $k$ be a field of characteristic~$p > 0$.  
\begin{enumerate}
\item Let $q$ be a power of $p$. Then $\PGL_2(k)$ contains subgroups
  isomorphic to $\PGL_2(\FF_q)$ and $\PSL_2(\FF_q)$ if and only if $\FF_q
  \subset k$.

\item Let $m \ge 1$ and $n \ge 1$ be integers with $n$ coprime to $p$. Let $e$
  be the order of $p$ in $(\ZZ / n\ZZ)^\times$. Then $\PGL_2(k)$ contains a
  $p$-semi-elementary subgroup of order $p^m n$ if and only if $\FF_{p^e}
  \subset k$, $e \mid m$, and $m \le \dim_{\FF_p}(k)$.\footnote{The condition
  $\FF_{p^e} \subset k$ is equivalent to the assertion that there is a primitive
  $n$-th root of unity in $k$.}

\item If $p = 2$ and $m \ge 1$ is an integer, then $\PGL_2(k)$ contains a
  non-split $2$-elementary subgroup of order $2^m$ if and only if $k$ contains a
  non-square element.
						
\item If $p = 2$ and $n > 1$ is odd, then $\PGL_2(k)$ contains a dihedral
  subgroup $\mathfrak{D}_n$ if and only if $\zeta + \zeta^{-1} \in k$ for some
  primitive $n$-th root of unity $\zeta$.
		
\item If $p = 3$, then $\PGL_2(k)$ contains a subgroup isomorphic to
  $\mathfrak{A}_5$ if and only if $\FF_9 \subset k$.
\end{enumerate}
\end{thmA*}	

\begin{thmB*}[Classification of Finite $p$-irregular subgroups]
Let $k$ be a field of characteristic $p > 0$. 
\begin{enumerate}
\item Fix $q > 2$ a power of $p$ such that $\FF_q \subset k$.  There is exactly
  one conjugacy class of subgroups isomorphic to each of $\PSL_2(\FF_q)$ and
  $\PGL_2(\FF_q)$.
					
\item Let $m,n \ge 1$ be integers with $n$ coprime to $p$, and suppose that $k$
  contains a primitive $n$-th root of unity. The conjugacy classes of split
  $p$-semi-elementary subgroups of order $p^m n$ are in bijection with the set
  of homothety classes\footnote{Two subgroups $\Gamma, \Gamma' \subset k$ are
    \textbf{homothetic} if $\Gamma' = \alpha \Gamma$ for some $\alpha \in
    k^\times$.}  of rank-$m$ subgroups $\Gamma \subset k$ that are stable under
  multiplication by elements of $\mu_n(k)$. The correspondence is
\[
\Gamma \mapsto \bp \mu_n(k) & \Gamma \\ & 1 \ep.
\]

\item Suppose that $p = 2$, and let $m \ge 1$ be an integer. The conjugacy
  classes of non-split $2$-elementary subgroups of order $2^m$ are parameterized
  by pairs $(k(\tau), G)$, where $k(\tau)$ is a quadratic inseparable extension
  of $k$ and $G$ is a subgroup of order $2^m$ of the abelian group
  \[
 \{I\} \cup
  \left\{ \bp \alpha & \tau^2 \\ 1 & \alpha \ep \ : \ \alpha \in k
  \smallsetminus \{\tau\} \right\}.
  \]
						
\item\label{item:p2} Suppose that $p = 2$ and $n > 1$ is an odd integer such that $\lambda:=
  \zeta + \zeta^{-1} \in k$ for some primitive $n$-th root of unity $\zeta$.
  Let $\mathfrak{Dih}_n(k)$ denote the set of conjugacy classes of dihedral
  subgroups of $\PGL_2(k)$ of order $2n$. The map $\mathfrak{Dih}_n(k) \to
  k^\times / (k^\times)^2$ defined by $G \mapsto \overline{\det}(t)$ for any
  involution $t \in G$ is well defined and injective. It is surjective if $k$
  contains a primitive $n$-th root of unity.
					  
\item If $\FF_9 \subset k$, then there is exactly one conjugacy class of
  subgroups isomorphic to $\mathfrak{A}_5$.
\end{enumerate}
Any $p$-irregular subgroup of $\PGL_2(k)$ is among the five types listed here.
\end{thmB*}	

\begin{remark}
Over an algebraically closed field $k$, dihedral and $p$-semi-elementary
subgroups may be characterized geometrically as follows. A subgroup of
$\PGL_2(k)$ ($p$-regular or otherwise) is dihedral if and only if it stabilizes
a pair of distinct points of $\PP^1(k)$, but does not fix them. A subgroup is
$p$-semi-elementary if it fixes a unique point of $\PP^1(k)$.
\end{remark}

\begin{remark}
  The map $\mathfrak{Dih}_n(k) \to k^\times / (k^\times)^2$ from Theorem~B\eqref{item:p2} is
  not surjective in general. For example, taking $k = \FF_2(T)$ and $n = 3$, one
  can show that the class of $T$ does not lie in the image. 
\end{remark}


The methods used to prove Theorem~B allow us, with essentially no extra work, to
give an elementary classification of the finite $p$-regular subgroups of
$\PGL_2(k)$ up to conjugation when $k$ is separably closed. By elementary, we
mean that it avoids the use of representation theory; see
\cite{Beauville_Finite_Subgroups_2010} for the representation theory approach
and for analogues of Theorems~A and~B for $p$-regular finite subgroups.

\begin{thmC*}[Finite $p$-regular subgroups]
Let $k$ be a separably closed field, and let $G$ be a finite subgroup of
$\PGL_2(k)$ such that $p \nmid |G|$. Then up to conjugation, $G$ is one of the
following subgroups:
\begin{enumerate}
\item $G = \langle \mat{\zeta & \\ & 1} \rangle$ for some $\zeta \in k^\times$;
  here $G$ is cyclic.
			
\item $G = \langle \mat{\zeta & \\ & 1}\rangle \rtimes \langle \mat{ & 1 \\ 1 &}
  \rangle$ for some $\zeta \in k^\times$; here $G$ is dihedral.
			
\item\label{Item: algebraic A_4} $G = N \rtimes C \cong \mathfrak{A}_4$, where
  $N = \left\{ \mat{\pm 1 & \\ & 1}, \mat{& \pm 1\\1&}\right\}$ and $C = \left\{
  I, \mat{1 & i \\ 1 & -i}, \mat{ 1 & 1 \\ -i & i}\right\}$, where $i$ is any
  primitive fourth root of unity.

\item $G = \left\langle T, \mat{i & \\ & 1} \right\rangle \cong
  \mathfrak{S}_4$, where $T = N \rtimes C \cong \mathfrak{A}_4$ is the group in
  \eqref{Item: algebraic A_4}.
		
\item $G = \langle s, t \rangle \cong \mathfrak{A}_5$, where $s = \mat{\lambda &
  \\ & 1}$, $t = \mat{ 1 & 1 - \lambda - \lambda^{-1} \\ 1 & -1}$, and $\lambda$
  is any primitive fifth root of unity. These generators satisfy
  $s^5 = t^2 = (st)^3 = I$.
\end{enumerate}	
In particular, if $G, G' \subset \PGL_2(k)$ are (abstractly) isomorphic finite
groups, then they are conjugate.
\end{thmC*}

\begin{remark}
The symmetric and alternating subgroups that occur in the classification of
$p$-regular subgroups (Theorem~C) also appear in Theorems~A and~B, albeit in
disguise. When $p = 2$, the subgroups isomorphic to $\mathfrak{A}_4$ are all
conjugate to the $p$-semi-elementary subgroup $B(\FF_4) = \mat{\FF_4^\times &
  \FF_4 \\ & 1}$; a subgroup isomorphic to $\mathfrak{A}_5$ is necessarily
conjugate to $\PGL_2(\FF_4)$; and there is no finite subgroup of $\PGL_2(k)$
isomorphic to $\mathfrak{S}_4$. When $p = 3$, any subgroup isomorphic to
$\mathfrak{A}_4$ (resp. $\mathfrak{S}_4$) is conjugate to $\PSL_2(\FF_3)$
(resp. $\PGL_2(\FF_3)$).  When $p = 5$, any subgroup isomorphic to
$\mathfrak{A}_5$ is conjugate to $\PSL_2(\FF_5)$.
\end{remark}

\begin{remark}
  For the classification of $p$-regular subgroups up to conjugation over an
  arbitrary field, see \cite{Beauville_Finite_Subgroups_2010}. We summarize
  these results in \S\ref{sec:arithmetic_p_regular}.
\end{remark}

Finally, as a concrete application of these results we characterize all
subgroups of the projective linear group over a finite field. A $p$-regular
cyclic subgroup of $\PGL_2(\FF_q)$ is said to be \textbf{split} if it fixes two
$\FF_q$-rational points of $\PP^1(\FF_q)$. It is called \textbf{non-split} if it
fixes a pair of quadratic conjugate points of $\PP^1(\FF_{q^2})$. A $p$-regular
dihedral subgroup of $\PGL_2(\FF_q)$ will be called split or non-split depending
on whether its normal cyclic subgroup is split or non-split.

\begin{thmD*}
  Let $\FF_q$ be a finite field with $q = p^r$, and write $G = \PGL_2(\FF_q)$.
  Each conjugacy class of nontrivial subgroups of $G$ is described by one of the
  following ten cases:
\begin{enumerate}
\item (Split cyclic) If $n \ge 2$ satisfies $q \equiv 1 \pmod n$, then $G$
  contains a unique conjugacy class of split cyclic subgroups of order~$n$.

\item (Non-split cyclic) If $n \ge 2$ satisfies $q \equiv -1 \pmod n$, then $G$
  contains a unique conjugacy class of non-split cyclic subgroups of order~$n$.

\item (Split dihedral) Suppose $n \ge 3$ satisfies $q \equiv 1 \pmod n$. If $q
  \equiv 1 \pmod{2n}$, then $G$ contains two conjugacy classes of split dihedral
  subgroups of order $2n$; otherwise, it contains one such conjugacy class.

\item (Non-split dihedral) Suppose $n \ge 3$ satisfies $q \equiv -1 \pmod n$. If
  $q \equiv -1 \pmod{2n}$, then $G$ contains two conjugacy classes of non-split
  dihedral subgroups of order $2n$; otherwise, it contains one such conjugacy
  class.
  
\item ($4$-groups) If $q$ is odd, then $G$ contains exactly two conjugacy classes of
  subgroups isomorphic to $(\ZZ / 2\ZZ)^2$.
  
\item If $p$ is odd, or if $p = 2$ and $r$ is even, then $G$ contains a unique
  conjugacy class of subgroups isomorphic to $\mathfrak{A}_4$.
		
\item If $p \neq 2$, then $G$ contains a unique conjugacy class of subgroups
  isomorphic to $\mathfrak{S}_4$.
		
\item If $q \equiv 0, \pm 1 \pmod 5$, then $G$ contains a unique conjugacy class
  of subgroups isomorphic to $\mathfrak{A}_5$.
			
\item If $s \mid r$, then $G$ contains a unique conjugacy class of subgroups
  isomorphic to $\PGL_2(\FF_{p^s})$ and $\PSL_2(\FF_{p^s})$.

\item If $m, n$ are positive integers with $m \le r$ and $n$ coprime to $p$, and
  if $e \mid \gcd(r,m)$, where $e$ is the order of $p$ in $(\ZZ / n\ZZ)^\times$,
  then $G$ contains $p$-semi-elementary subgroups of order $p^m n$. The
  conjugacy classes of such subgroups are in bijection with the set of homothety
  classes of $\FF_{p^e}$-vector subspaces of $\FF_q$ of dimension $m/e$.
\end{enumerate}
\end{thmD*}


\section{Fixed points}
\label{Sec: Fixed Points}

\begin{convention*} 
Throughout this section, $k$ is an algebraically closed field of characteristic
$p \geq 0$.
\end{convention*}

\bigskip

The fixed points of an element $s \in \PGL_2(k)$ will be of paramount interest
in our study, and it will often be convenient to relocate them by a suitable
conjugation.  More precisely, if $s$ fixes the point $x \in \PP^1(k)$, and if $t
\in \PGL_2(k)$ satisfies $t.x = y$, then $tst^{-1}$ fixes $y$. We say that we
have ``conjugated the fixed point $x$ to~$y$''.

If $s =\mat{\alpha & \beta \\ \gamma & \delta } \in \PGL_2(k)$ is nontrivial,
then $s$ has either one or two distinct fixed points in $\PP^1(k)$. Indeed, the
fixed point equation $s.z = z$ is equivalent to
\begin{equation}
\label{Eq: Fixed points}
  \gamma z^2 + (\delta - \alpha)z - \beta = 0. 
\end{equation}
If $\gamma = 0$, then $z = \infty$ is fixed, and the fixed point equation has at
most one solution. Otherwise, the fixed point equation has at most two
solutions. An immediate consquence of these observations is the following well
known fact about $\PGL_2(k)$:

\begin{proposition}
  Let $s \in \PGL_2(k)$. If $s$ fixes three distinct points of $\PP^1(k)$, then
  $s$ is the identity.
\end{proposition}
	
Assume now that $s$ has finite order. If $s$ has a unique fixed point, we may
conjugate it to~$\infty$ and see that $s.z = z + \beta$, or in matrix form: $s =
\mat{1 & \beta \\ & 1}$. Since $s^m = \mat{1 & m\beta \\ & 1}$, we find $s$ has
order $p > 0$.
	
Conversely, if $s \in \PGL_2(k)$ has order $p > 0$, then we claim that $s$ has a
unique fixed point. For suppose $s$ has two distinct fixed points, and let us
conjugate them to~$0$ and~$\infty$. Then $s.z = \alpha z$, so that $\alpha^p =
1$; hence, $\alpha = 1$. This contradiction shows that $s$ must have a unique
fixed point.
		
\begin{lemma}
\label{Lem: order p}
Let $s = \mat{ \alpha & \beta \\ \gamma & \delta} \in \PGL_2(k)$ be nontrivial
and of finite order. Write $\tr(s)$ for the trace of $s$ (which is well defined
up to a scalar multiple). The following are equivalent:
\begin{enumerate}
\item  $s$ has a unique fixed point in $\PP^1(k)$;
\item  $s$ has order~$p > 0$; and
\item $\tr(s)^2 = 4\det(s)$.
\end{enumerate}
\end{lemma}

\begin{proof}
The equivalence of the first two statements was proved above. We now prove the
equivalence of the first and third statements. The final statement is
homogeneous and quadratic in the entries of the matrix $s$, so it is well
defined on $\PGL_2(k)$. Moreover, the first and third statements are invariant
under conjugation in $\PGL_2(k)$, so we may assume that $s$ fixes $\infty$. Thus
$\gamma = 0$. Now $\infty$ is the unique fixed point of $s$ precisely when
$\delta - \alpha = 0$ by~\eqref{Eq: Fixed points}, or equivalently, when
\[
	\tr(s)^2 - 4\det(s) = (\alpha + \delta)^2 - 4\alpha \delta = (\delta - \alpha)^2 = 0. \qedhere
\]
\end{proof}

We say that $s$ is \textbf{unipotent} if it satisfies the three equivalent
conditions in the lemma. Note that a unipotent element can be conjugated to
$\mat{1 & \beta \\ & 1}$.
	
\begin{lemma}
\label{Lem: order n}
Let $s  = \mat{ \alpha & \beta \\ \gamma & \delta} \in \PGL_2(k)$ be nontrivial. 
\begin{itemize}
\item $s$ has order~2 if and only if $\tr(s) = 0$.
\item $s$ has order~3 if and only if $\tr(s)^2 = \det(s)$.
\item $s$ has order~5 if and only if $\tr(s)^4 - 3\tr(s)^2\det(s) + \det(s)^2 = 0$. 
\end{itemize}
\end{lemma}

\begin{remark}
Note that these trace/determinant equations are homogeneous in the entries of
$s$ of degree 1, 2, and 4, respectively, so that their solutions are well
defined elements of $\PGL_2(k)$.
\end{remark}

\begin{remark}
Using Chebychev polynomials, one can formulate trace/determinant equations like
this to characterize the elements of any order.
\end{remark}

\begin{proof}
In all cases, the conditions are invariant under conjugation, so we may assume
that $s$ fixes~$\infty$. If $\infty$ is the only fixed point, then the previous
lemma shows that $s = \mat{1 & \gamma \\ & 1}$ for some nonzero $\gamma \in k$
and $s$ has order~$p$. For $p = 2, 3$, the third statement of the previous lemma
reduces immediately to the desired equations for the trace and determinant of
$s$. For $p = 5$, we see that
\[
\tr(s)^4 - 3\tr(s)^2\det(s) + \det(s)^2 = \left[ \tr(s)^2 - 4 \det(s) \right]^2,
\]
so that the desired criterion for order~5 reduces to the one in the previous
lemma.
	
Suppose now that $s$ has two distinct fixed points. Then we may suppose after
conjugation that it fixes $0$ and $\infty$. Hence $s = \mat{\lambda & \\ & 1}$,
and $s$ has order~$n$ precisely when $\lambda$ is a primitive $n$-th root of
unity. For $n = 2$, this means $\lambda = -1$, so that $\tr(s) = 0$. For $n =
3$, this means $\lambda^2 + \lambda + 1 = 0$, so that
\[
	\tr(s)^2 - \det(s) = (\lambda + 1)^2 - \lambda = 0.
\]
For $n = 5$, it means $\lambda^4 + \lambda^3 + \lambda^2 + \lambda + 1 = 0$, so
that
\[
	\tr(s)^4 -3\tr(s)^2\det(s) + \det(s)^2 = (\lambda + 1)^4 - 3\lambda (\lambda + 1)^2 + \lambda^2 = 0.\qedhere
\]
\end{proof}


\section{Special subgroups}
\label{Sec: Special Subgroups}
	
\begin{convention*}
Throughout this section, $k$ is an algebraically closed field unless otherwise
specified.
\end{convention*}

	
\subsection{Cyclic subgroups}
\label{Sec: Cyclic}

If $s \in \PGL_2(k)$ is nontrivial, then it fixes at least one point of
$\PP^1(k)$, which we may assume is $\infty$ after a suitable conjugation. If
$\infty$ is the only fixed point of $s$, then $s = \mat{1 & \beta \\ & 1}$, and
it generates a cyclic group of order~$p$ (Lemma~\ref{Lem: order p}), where
$\mathrm{char}(k) = p$. Moreover, we see that $\mat{\beta^{-1} & \\ & 1} s
\mat{\beta & \\ & 1} = \mat{1 & 1 \\ & 1}$, and hence every cyclic subgroup of
order $p$ is conjugate to the subgroup $\langle \mat{1 & 1 \\ & 1} \rangle$
inside $\PGL_2(k)$.

Now suppose that $s$ fixes two distinct points of $\PP^1(k)$. After a suitable
conjugation, we may assume that $s$ fixes $0$ and $\infty$. Hence $s =
\mat{\alpha & \\ & 1}$. If $\alpha^m \neq 1$ for any $m \neq 0$, then $s$ has
infinite order and generates a subgroup isomorphic to $\ZZ$. If $\alpha^m = 1$
for some $m > 1$, then $\alpha$ is a root of unity in~$k$. The roots of unity
in~$k$ have order prime to the characteristic.  Hence $s$ generates a
$p$-regular cyclic subgroup of $\PGL_2(k)$.

\begin{proposition}
\label{Prop: Cyclic}
Let $G$ be a nontrivial finite cyclic subgroup of $\PGL_2(k)$. Then exactly one
of the following is true:
\begin{enumerate}
\item $|G| = p$, it fixes a unique point of $\PP^1(k)$, and $G$ is conjugate to
  $\mat{1 & \FF_p \\ & 1}$; or
\item $G$ is $p$-regular, it fixes exactly two points of $\PP^1(k)$, and $G$
  is conjugate to $\mat{\Lambda & \\& 1}$ for some (cyclic) subgroup $\Lambda$
  of the roots of unity of $k$.
\end{enumerate}
In particular, the order of a finite cyclic group uniquely determines its
conjugacy class in $\PGL_2(k)$.
\end{proposition}

\begin{proof}
We have already proved everything but the final statement about conjugacy
classes. If $|G| = p$, then it is conjugate to $\mat{1 & \FF_p \\ & 1}$, so that
any cyclic group of order $p$ lies in the same conjugacy class. If $p \nmid
|G|$, then it is conjugate to $\mat{\Lambda & \\ & 1}$ for some finite cyclic
group $\Lambda \subset k^\times$. The group of roots of unity in $k^\times$
contains a unique cyclic subgroup of every order prime to the
characteristic. That is, the order of $G$ uniquely determines $\Lambda$, and
hence also the conjugacy class of $G$.
\end{proof}
	

Another normal form for $p$-regular cyclic subgroups will be useful when we pass
to non-algebraically closed fields.
	
\begin{corollary}
\label{Cor: Alternate Cyclic}
Let $G$ be a finite $p$-regular cyclic subgroup of $\PGL_2(k)$ of order $n \geq
3$, and let $\zeta$ be a primitive $n$-th root of unity. Then $G$ is conjugate
to the subgroup generated by $\mat{ \lambda + 1 & -1 \\ 1 & 1}$, where $\lambda
= \zeta + \zeta^{-1}$.
\end{corollary}

\begin{proof}
The elements $\mat{ \lambda + 1 & -1 \\ 1 & 1}$ and $\mat{ \zeta & \\ & 1}$ are
conjugate:
\[
\bp 1 & - \zeta^{-1} \\ 1 & - \zeta \ep^{-1} \bp \zeta & \\ & 1 \ep \bp 1 & - \zeta^{-1} \\ 1 & - \zeta \ep
	= \bp \lambda + 1 & -1 	 \\ 1 & 1 \ep. \qedhere
\]
\end{proof}


\subsection{\texorpdfstring{$p$}-subgroups}
\label{Sec: p-groups}
	
In this section, we will assume $\mathrm{char}(k) = p > 0$. Suppose $G \subset
\PGL_2(k)$ is a nontrivial $p$-group, and let $s \in G$ be any nontrivial
element. We saw in Lemma~\ref{Lem: order p} that $s$ must fix a unique point of
$\PP^1(k)$, so after conjugating $G$ if necessary, we may assume that $s$ fixes
$\infty$. We claim that every element of $G$ fixes $\infty$. Suppose not, and
select $s' \in G \smallsetminus\{I, s\}$ that fixes a (unique) point $x \neq
\infty$. After conjugating $G$ by $\mat{1 & x \\ & 1}$, we may assume that $s'$
fixes $0$ and $s$ still fixes $\infty$. So $s = \mat{1 & \beta \\ & 1}$ and $s'
= \mat{1 & \\ \beta' & 1}$ for some $\beta, \beta' \in k^\times$. Since $G$ is a
$p$-group, it follows that $ss'$ must fix a unique point as
well. Lemma~\ref{Lem: order p} shows that
\[
\tr(ss')^2 - 4\det(ss') = \beta \beta' ( 4 + \beta \beta') = 0.
\]
When $p = 2$, this contradicts $\beta \beta' \neq 0$. When $p > 2$, we find that
$s^{-1}s' \in G$ as well, so that
\[
	\tr(s^{-1} s') - 4\det(s^{-1} s') = -\beta \beta' (4 - \beta \beta') = 0.
\]
Adding these last two equations gives a contradiction. Hence every element of
$G$ fixes $\infty$. We have proved the following lemma.


\begin{lemma}
\label{Lem: p-group fixed point}
If $G \subset \PGL_2(k)$ is a nontrivial $p$-group, then $G$ is conjugate to
$\mat{1 &\Gamma \\ & 1}$, where $\Gamma$ is an additive
subgroup of $k$, and hence an $\FF_p$-vector space. In particular, $G$ is
abelian and fixes a unique point of $\PP^1(k)$.
\end{lemma}

 Now suppose that $G$ and $G'$ are two finite $p$-groups in $\PGL_2(k)$. To
 determine necessary and sufficient conditions for these two subgroups to be
 conjugate, it suffices to assume that $G = \mat{1 & \Gamma \\ & 1}$ and $G' =
 \mat{1 & \Gamma' \\ & 1}$ for some additive subgroups $\Gamma$ and $\Gamma'$
 inside $k$. If there exists $u \in \PGL_2(k)$ so that $uGu^{-1} = G'$, then $u$
 must fix $\infty$, so that $u = \mat{\alpha & \beta \\ & 1}$. Hence
 \[
 	uGu^{-1} = \begin{pmatrix}\alpha & \beta \\ & 1\end{pmatrix} 
		\begin{pmatrix}1 & \Gamma \\ & 1 \end{pmatrix}
		\begin{pmatrix}\alpha^{-1} & - \alpha^{-1} \beta \\ & 1 \end{pmatrix}
		= \begin{pmatrix}1 & \alpha \Gamma \\ & 1\end{pmatrix},
 \]
which implies that $\Gamma' = \alpha \Gamma$. This calculation also shows that
it is sufficient that $\alpha\Gamma = \Gamma'$ for some $\alpha \in k^\times$
in order to have $G$ and $G'$ be conjugate.

\begin{proposition}
The conjugacy classes of finite $p$-subgroups of $\PGL_2(k)$ are in bijective
correspondence with the finite additive subgroups of $k$ modulo homotheties.
\end{proposition}

To close this section, we define the \textbf{stabilizer}\footnote{The
  multiplicative group $k^\times$ acts on the set of all finite additive
  subgroups of $k$, and the stabilizer (in the usual sense of a group action) of
  a particular such subgroup $\Gamma$ is precisely $\FF_\Gamma^\times$. Dickson
  prefers to call $\FF_\Gamma$ the \textbf{multiplier}
  \cite[\S70]{Dickson_Linear_Groups_1901}.} of a finite additive subgroup
$\Gamma \subset k$ to be
\[
\FF_\Gamma = \{\alpha \in k : \alpha \Gamma \subset \Gamma\}.
\]
Then one verifies immediately that $\FF_\Gamma$ is a subfield of $k$. Moreover, each
nonzero element of $\FF_\Gamma$ induces an $\FF_p$-linear automorphism of
$\Gamma$, and since $\Gamma$ is finite, there are only finitely many
possibilities in total for such an automorphism. Hence $\FF_\Gamma$ is a finite
subfield of $k$. Observe further that for any $\alpha \in k^\times$, we have
$\FF_{\alpha \Gamma} = \FF_\Gamma$, so that the stabilizer is a homothety class
invariant.




\subsection{Subgroups stabilizing a pair of points}
\label{Sec: Stable pair of points}
	
Let $G$ be a finite subgroup of $\PGL_2(k)$ that stabilizes a pair of points,
but does not fix them. After conjugation, we may assume that $G$ stabilizes
$\{0, \infty\}$. Define
\[
	H = \{s \in G : s.\infty = \infty \text{ and } s.0 = 0\}.
\]
We saw in \S\ref{Sec: Cyclic} that $H$ is cyclic and generated by an element
$\mat{\lambda & \\ & 1}$. We observe that $H$ is normal in $G$. Indeed, it is
invariant under conjugation by any element fixing both $0$ and $\infty$ as these
are precisely the elements of $H$. Any element of $G$ that stabilizes the set
$\{0, \infty\}$ without fixing it pointwise must be of the form $t = \mat{ &
  \tau \\ 1 }$ for some $\tau \in k^\times$. We find that
\[
t\begin{pmatrix}\lambda & \\ & 1\end{pmatrix}t^{-1} = \begin{pmatrix}& \tau \\ 1 \end{pmatrix}
	\begin{pmatrix} \lambda & \\ & 1\end{pmatrix}\begin{pmatrix} & \tau \\ 1\end{pmatrix}
	= \begin{pmatrix} \lambda^{-1} & \\ & 1 \end{pmatrix}.
\]
Hence $H$ is normal. In fact, this computation also shows that the subgroup of
$G$ generated by $H$ and $t$ is dihedral.

We now prove that $G$ is generated by $H$ and $t$. We may conjugate $G$ by
$\mat{\sqrt{\tau^{-1}} & \\ & 1}$ in order to assume that $t = \mat{ & 1 \\ 1
  &}$; note that this operation does not affect~$H$.  We have already shown that
$G \smallsetminus H$ consists of elements of the form $\mat{ & \tau \\ 1
  &}$. Suppose we have such an element. Then $\mat{ & \tau \\ 1 &} \mat{ & 1
  \\ 1 &} = \mat{\tau & \\ & 1} \in H$. If $\lambda$ has order $n = |H|$, then
$\tau^n = 1$, which means there are at most $n$ such elements in $G$. On the
other hand, we can generate $n$ elements of this shape via
\[
	\bp \lambda^j & \\ & 1 \ep \bp & 1 \\ 1 & \ep = \bp & \lambda^j \\ 1 & \ep, \quad j = 1, \ldots, n.
\]
Hence $G$ is generated by $H$ and $t$ as desired.

\begin{proposition}
  \label{prop:dihedral}
Let $G \subset \PGL_2(k)$ be a finite subgroup that stabilizes a pair of points
of $\PP^1(k)$. Then up to $\PGL_2(k)$-conjugacy, $G$ satisfies one of the
following:
\begin{enumerate}
\item $G = \mat{ \Lambda & \\ & 1}$ for some cyclic subgroup $\Lambda \subset
  k^\times$; in this case, $G$ fixes a pair of points of $\PP^1(k)$.
\item $G = \mat{ \Lambda & \\ & 1} \rtimes \langle \mat{ & 1 \\ 1 & } \rangle$
  for some cyclic subgroup $\Lambda \subset k^\times$; in this case, $G$
  stabilizes a pair of points, but does not fix them.
\end{enumerate}
In either case, we observe that $G$ is uniquely determined up to
$\PGL_2(k)$-conjugacy by its order and whether it fixes a pair of points.
\end{proposition}

Now let us suppose that $G$ is a finite subgroup of $\PGL_2(k)$, and let $H$ be
a nontrivial maximal $p$-regular cyclic subgroup. We showed in
Proposition~\ref{Prop: Cyclic} that $H$ fixes a pair of points $\{x,y\} \subset
\PP^1(k)$. If $s \in G$ lies in the normalizer of $H$, then for each nontrivial
$h \in H$, there is $h' \in H$ such that $shs^{-1} = h'$. Observe that
\[
	h's.x = sh.x = s.x \qquad h's.y = sh.y = s.y,
\] 
so that $s.x$ and $s.y$ are fixed points of $h'$. This means $s$ stabilizes the
pair $\{x,y\}$. If $s$ fixes both of these points, then $s \in H$ by maximality;
otherwise, $s$ swaps $x$ and $y$. It follows that the normalizer $N_G(H)$
consists of the  elements of $G$ that stabilize the pair of points $\{x,y\}$, so
that our above work proves $N_G(H) = H$ or $N_G(H)$ is dihedral with maximal
cyclic subgroup $H$ (of index~2).

\begin{proposition}
\label{Prop: Dihedral normalizer}
Let $G$ be a finite subgroup of $\PGL_2(k)$, and let $H$ be a nontrivial maximal
$p$-regular cyclic subgroup. Then the normalizer of $H$ satisfies $[N_G(H):H ] =
1$ or $2$. In the latter case, $N_G(H)$ is dihedral.
\end{proposition}


\subsection{Subgroups fixing a unique point}
\label{Sec: Borel}

The goal of this section is to prove the following result:
	
\begin{proposition}
\label{Prop: Borel Normalization}
Suppose $G$ is a finite subgroup of $\PGL_2(k)$ that fixes a unique point of
$\PP^1(k)$. Then $k$ has positive characteristic $p$, and up to conjugation in
$\PGL_2(k)$, there exist a nontrivial additive subgroup $\Gamma \subset k$ and a
positive integer $n$ coprime to $p$ satisfying $\mu_n(k) \subset
\FF_\Gamma^\times$ and
\[
G = \begin{pmatrix} \mu_n(k) & \Gamma \\ & 1 \end{pmatrix}
  = \begin{pmatrix} 1 & \Gamma \\ & 1 \end{pmatrix} \rtimes 
	\begin{pmatrix} \mu_n(k) & \\ & 1 \end{pmatrix}.
\]
The group $G$ is determined up to $\PGL_2(k)$-conjugation by $n$ and the
homothety class $\{\alpha\Gamma : \alpha \in k^\times\}$.
\end{proposition}

\begin{remark}
\label{Rem: Cyclotomic Field}
With the notation of the proposition, let $e$ be the order of $p$ in $(\ZZ /
n\ZZ)^\times$. Then $\FF_{p^e}$ is the smallest extension of $\FF_p$ containing
$\mu_n(k)$, and hence $\FF_{p^e} \subset \FF_\Gamma$.
\end{remark}

\begin{corollary}
If $G$ is a finite subgroup of $\PGL_2(k)$, then $G$ fixes a unique point of
$\PP^1(k)$ if and only if $G$ is $p$-semi-elementary with nontrivial Sylow
$p$-subgroup, where $p = \mathrm{char}(k) > 0$.
\end{corollary}

\begin{proof}
One implication follows immediately from the preceding proposition. For the
other, $G$ is $p$-semi-elementary if and only if it fits into an exact sequence
$1 \to P \to G \to G / P \to 1$ with $P$ a $p$-group and $G / P$ cyclic of order
prime to $p$. We saw in \S\ref{Sec: p-groups} that $P$ must fix a unique point
of $\PP^1(k)$, and so must any group that normalizes it.
\end{proof}



Before proving the proposition, we discuss Borel and unipotent subgroups. Write
$B(k)$ for the \textbf{standard Borel subgroup} of $\PGL_2(k)$:
\[
B(k) = \{s \in \PGL_2(k) : s.\infty = \infty\} = \left\{ \begin{pmatrix} \alpha & \beta \\ & 1 \end{pmatrix} : 
	\alpha \in k^\times, \beta \in k \right\}.
\]
Write $U(k)$ for the unipotent subgroup of $B(k)$ --- i.e., those
elements with $\alpha = 1$. Note that $U(k)$ is an abelian group, and it can be
written concretely as $U(k) = \mat{1 & k \\ & 1}$. Moreover, an immediate
calculation shows that $U(k)$ is a normal subgroup of $B(k)$. 

More generally, any subgroup conjugate to $B(k)$ will be called a Borel
subgroup. Equivalently, a Borel subgroup may be characterized as the set of all
elements of $\PGL_2(k)$ fixing a particular point of $\PP^1(k)$.

There is an exact sequence of homomorphisms
\[
	1 \rightarrow U(k) \rightarrow B(k) \stackrel{\pi}{\rightarrow} k^\times \to 1,
\]
where $\pi$ maps an element $s = \mat{ \alpha & \beta \\ & 1}$ to the derivative
of $s.z = \alpha z + \beta$. (Note that this is well defined independent of the
matrix representation of $s$.) For any subgroup $G \subset \PGL_2(k)$, we write
$B_G = B(k) \cap G$ and $U_G = U(k) \cap G = \mat{1 & \Gamma \\ & 1}$ for some
additive subgroup $\Gamma \subset k$. Then the above exact sequence descends to
an exact sequence
\[
	1 \rightarrow U_G \rightarrow B_G \stackrel{\pi}{\rightarrow} \pi(B_G) \to 1.
\]
If $\pi(B_G)$ is cyclic, generated by $\lambda \in k^\times$, then there exists
$s = \mat{\lambda & \eta \\ & 1} \in B_G$ for some $\eta \in k$. The subgroup
generated by $s$ is isomorphic to $\pi(B_G)$, so we may represent $B_G$ as
a semidirect product:
\[
	B_G = \begin{pmatrix} 1 & \Gamma \\ & 1 \end{pmatrix} \rtimes 
	\langle s \rangle.
\]
This is the case, for example, if $G$ is a finite group, so that $\pi(B_G)
\subset \mu_n(k)$ for some $n \geq 1$. Note that if $s$ is nontrivial, then
$\lambda \neq 1$.

\begin{proof}[Proof of Proposition~\ref{Prop: Borel Normalization}]
Suppose now that $G$ is a finite subgroup of $\PGL_2(k)$ that fixes a unique
point of $\PP^1(k)$. As usual, we may assume that $G$ fixes $\infty$ after a
suitable conjugation, so that $G \subset B(k)$. Moreover, let us write $U_G =
\mat{1 & \Gamma \\ & 1}$ and choose an element $s$ as above so that $G = \mat{1
  & \Gamma \\ & 1} \rtimes \langle s \rangle$. Now $s$ fixes $\infty$ and at
least one other point of $\PP^1(k)$.  Note that $\Gamma \neq 0$, else $G$ fixes
at least two points. In particular, the characteristic of $k$ is positive. After
conjugating $G$ by an element of $U(k)$, we may assume that $s$ fixes 0, so that
it is of the form $s = \mat{\lambda & \\ & 1}$. That is, $\langle \lambda
\rangle = \mu_n(k)$ for some $n$ coprime to $p$, and $G = \mat{ 1 & \Gamma \\ &
  1 } \rtimes \mat{\mu_n(k) & \\ & 1} = \mat{ \mu_n(k) & \Gamma \\ & 1}$.

Now observe that, with $\lambda$ as above and $\gamma \in \Gamma$, we have
\[
	\bp \lambda &  \\ & 1 \ep
	\bp 1 & \gamma \\ & 1 \ep
	\bp \lambda^{-1} & \\ & 1 \ep =
	\bp 1 & \lambda \gamma \\ & 1 \ep \in G,
\]
so that $\lambda\Gamma \subset \Gamma$. That is, $\mu_n(k) \subset
\FF_\Gamma^\times$.

Finally we must deal with the question of conjugacy of these subgroups. Let $G,
G' \subset \PGL_2(k)$ be finite subgroups of the following form:
\[
	G = \begin{pmatrix} 1 & \Gamma \\ & 1 \end{pmatrix} \rtimes 
		\begin{pmatrix} \mu_n(k) & \\ & 1 \end{pmatrix}  \qquad
	G' = \begin{pmatrix} 1 & \Gamma' \\ & 1 \end{pmatrix} \rtimes 
		\begin{pmatrix} \mu_{n'}(k) & \\ & 1 \end{pmatrix} ,
\]
where $n, n' \in \NN \smallsetminus p\NN$ and $\Gamma, \Gamma' \subset k$ are
finite nontrivial additive subgroups. Suppose first that there is $s \in
\PGL_2(k)$ such that $sGs^{-1} = G'$. Then $s$ must fix $\infty$, so that $s =
\mat{ \alpha & \beta \\ & 1}$.  Since $sU_Gs^{-1} = U_{G'}$, our work in
\S\ref{Sec: p-groups} shows that $\Gamma' = \alpha \Gamma$. Comparing the orders
of $G$ and $G'$ shows $n = n'$.

Conversely, suppose that $G$ and $G'$ are as above, that $\Gamma' = \alpha
\Gamma$, and $n = n'$. For $\gamma \in \Gamma$ and $\lambda \in \mu_n(k)$, we
have
\[
\begin{pmatrix} \alpha & \\ & 1 \end{pmatrix} \begin{pmatrix} \lambda & \gamma \\ & 1 \end{pmatrix}
	\begin{pmatrix} \alpha^{-1} & \\ & 1 \end{pmatrix} 
= \begin{pmatrix} \lambda & \alpha \gamma \\ & 1 \end{pmatrix} \in G'.
\]
As $G$ and $G'$ have the same order, it follows that $G$ and $G'$ are
conjugate. 
\end{proof}

When we deal with non-algebraically closed fields later, it will be useful to
have a description of $2$-elementary subgroups that fix a point other than
infinity.

\begin{proposition}
  \label{prop:alternate-elementary}
  Suppose that $k$ has characteristic~2. Let $\tau \in \PP^1(k) \smallsetminus
  \infty$. The subgroup of $\PGL_2(k)$ given by
  \[
  \Omega(\tau) := \{I\} \cup \left\{
  \begin{pmatrix}  \alpha & \tau^2 \\ 1 & \alpha 
  \end{pmatrix} \ : \ \alpha \in k \smallsetminus \{\tau\}\right\}  
  \]
  fixes $\tau$, and any element of $\PGL_2(k)$ with unique fixed point $\tau$
  lies in $\Omega(\tau)$.
\end{proposition}

\begin{proof}
  It is immediate that $\Omega(\tau)$ fixes $\tau$. If $s = \mat{\alpha & \beta
    \\ \gamma & \delta}$ is a nontrivial element of $\PGL_2(k)$ with unique
  fixed point $\tau$, then $s$ has order~2 and trace zero (Lemma~\ref{Lem: order
    p}). It follows that $\delta = \alpha$. As $s$ does not fix $\infty$, we
  must have $\gamma \ne 0$. Without loss, set $\gamma = 1$. The fixed point of
  $s = \mat{\alpha & \beta \\ 1 & \alpha}$ is $\sqrt{\beta} = \tau$. Hence,
  $\beta = \tau^2$, as desired. 
\end{proof}


\subsection{Tetrahedral subgroups}
\label{Sec: Tetrahedral}

Recall that the group of orientation-preserving symmetries of a regular
tetrahedron is isomorphic to $\mathfrak{A}_4$. (Look at the action of the
symmetry group on the four vertices of the tetrahedron.)  Any group isomorphic
to $\mathfrak{A}_4$ will therefore be called \textbf{tetrahedral}.

\begin{lemma}
\label{Lem: Tetrahedral}
Let $G$ be a non-abelian group of order $12$ possessing a normal Klein
4-subgroup. Then $G$ is tetrahedral.
\end{lemma}

\begin{proof}
Let $N = \{e, n_1, n_2, n_3\}$ be the given normal subgroup, let $h \in G$ be an
element of order~3, and let $H = \langle h \rangle$. Then $N \cap H = \{e\}$ and
$NH = G$. We observe that $hNh^{-1} = N$, so that $hn_ih^{-1} = n_j$ for some
$j$. If conjugation by $h$ fixes all $n_i$, then we would find that $G$ is
abelian. If conjugation by $h$ fixed only one $n_i$ and permuted the other~2,
then $h$ would have order~2. So $h$ permutes the $n_i$ cyclically. Hence $G = N
\rtimes H$, and the action $H \to \Aut(N)$ is given by cyclic permutation on the
three nontrivial elements of $N$. One now checks that $\mathfrak{A}_4$ may also
be written as a semidirect product of the normal subgroup $\{e, (12)(34),
(13)(24), (14)(23)\}$ (containing all elements of order 2) and the subgroup
generated by $(123)$; hence, it is isomorphic to $G$.
\end{proof}


\begin{proposition}
\label{Prop: Tetrahedral 2}
If $p = 2$ and $G$ is a tetrahedral subgroup of $\PGL_2(k)$, then $G$ is
conjugate to the standard Borel subgroup $B(\FF_4) = \mat{1 & \FF_4 \\ & 1}
\rtimes \mat{\FF_4^\times & \\ & 1}$.
\end{proposition}

\begin{proof}
We know $G$ contains a normal 4-group $N$, each nontrivial element of which must
be unipotent since $p = 2$. After conjugating $G$ if needed, we may assume that
$N = \mat{1 & \Gamma \\ & 1}$ for some $\Gamma \subset k$ of rank~2
(Lemma~\ref{Lem: p-group fixed point}). By normality, if $s \in G$ and $u \in
N$, then there is $u' \in N$ such that $su = u's$. Since $\infty$ is the unique
fixed point of each nontrivial element of $N$, we have
\[
s.\infty = s(u.\infty) = u'(s.\infty) \ \Rightarrow \ s.\infty = \infty.
\]
That is, $G$ fixes $\infty$, or equivalently $G \subset B(k)$.

By Proposition~\ref{Prop: Borel Normalization}, we may assume $G = \mat{1 &
  \Gamma \\ & 1} \rtimes \mat{\mu_3(k) & \\ & 1}$, where $\Gamma$ is an additive
subgroup of $k$ of order~4 that is stable under multiplication by $\mu_3(k) =
\FF_4^\times$. For $\gamma \in \Gamma \smallsetminus \{0\}$, observe that
\[
	\bp  \gamma^{-1} & \\ & 1\ep  \bp 1 &\Gamma \\ & 1 \ep \bp \gamma & \\ & 1 \ep 
		= \bp 1 & \gamma^{-1} \Gamma \\ & 1 \ep. 
\]
So after an appropriate conjugation, we may assume that $1 \in \Gamma$. Since
$\Gamma$ is stable under multiplication by elements of $\FF_4^\times$, we
conclude that $\Gamma = \FF_4$.
\end{proof}

\begin{proposition}
\label{Prop: Tetrahedral p}
Suppose $p \ne 2$ and let $G$ be a tetrahedral subgroup of $\PGL_2(k)$. Then $G$
is conjugate to the semidirect product $N \rtimes C$, where $N = \left\{ \mat{
  \pm 1 & \\ & 1 }, \mat{ & \pm 1\\ 1 & } \right\}$, and $C$ is the cyclic group
of order~$3$ generated by $\mat{ 1 & i \\ \\ 1 & -i}$, where $i$ is any choice
of primitive fourth root of unity. In particular, any two tetrahedral subgroups
of $\PGL_2(k)$ are conjugate when $\mathrm{char}(k)$ is different from~2.
\end{proposition}


\begin{proof}
A tetrahedral group contains a normal 4-group $N$. Let $s_1$ be a nontrivial
element of $N$, and let us conjugate $G$ so that $s_1$ fixes $0$ and
$\infty$. (Here we have used the hypothesis $p \neq 2$.) If $s_2 \in N$ is
another element of order~2, then it must commute with $s_1$, so that
\[
s_1s_2.0 = s_2s_1.0 = s_2.0 \qquad s_1s_2.\infty = s_2s_1.\infty = s_2.\infty.
\]
Hence $s_2$ stabilizes the set $\{0, \infty\}$, and it cannot fix these elements
since $s_1$ is the only element of order~2 with this property. Repeating this
argument for the third nontrivial element of $N$ shows that there exist $\tau
\neq \tau' \in k^\times$ such that
\[
N = \left\{I, \begin{pmatrix} -1 & \\ & 1 \end{pmatrix},
	\begin{pmatrix}  & \tau\\ 1&  \end{pmatrix},
	\begin{pmatrix}  & \tau'\\ 1 & \end{pmatrix} \right\}.
\]
After conjugating by $\mat{\sqrt{\tau^{-1}} & \\ & 1}$, we may assume that $\tau
= 1$. Since $N$ is abelian, we have
\[
	\bp 1 & \\ & \tau' \ep = 
	\bp & 1 \\ 1 & \ep 
	\bp & \tau' \\ 1 & \ep =
	\bp & \tau' \\ 1 & \ep
	\bp & 1 \\ 1 & \ep  = 
	\bp \tau' &  \\  & 1\ep .
\]
Hence $(\tau')^2 = 1$, or $\tau' = -1$, and $N$ is of the form claimed in the
statement of the proposition.

The group $G$ has normal subgroup $N$ and four conjugate subgroups of
order~3. In particular, every element of $G \smallsetminus N$ has order~3. We
now compute the set of all elements of $\PGL_2(k)$ of order~3 that normalize
$N$. Suppose $s = \mat{\alpha & \beta \\ \gamma & \delta}$ is such an
element. Then
\[
\bp \alpha & \beta \\ \gamma & \delta \ep
\bp -1 & \\ & 1 \ep
\bp \delta & - \beta \\ - \gamma & \alpha \ep
= \bp -(\alpha \delta + \beta \gamma) & 2\alpha \beta \\ -2\gamma \delta & \alpha \delta + \beta \gamma \ep.
\]
For $s$ to be a normalizing element, we must have either
\begin{equation}
\label{Eq: Case 1}
	\alpha \beta = \gamma\delta = 0, \text{ or}
\end{equation}
\begin{equation}
\label{Eq: Case 2}
	\alpha \delta = - \beta \gamma \quad \text{and} \quad \alpha \beta = 
	\pm \gamma \delta \quad \text{and} \quad \alpha \beta \gamma \delta \neq 0.
\end{equation}
Let us suppose first that \eqref{Eq: Case 1} holds. Then either $\alpha = \delta
= 0$ or $\beta = \gamma = 0$. In the former case, $s$ has order~3 if and only if
$\tr(s)^2 - \det(s) = - \det(s) = 0$, so that this cannot occur (Lemma~\ref{Lem:
  order n}). In the latter case, we may assume $\delta = 1$, so that $s$ has
order~3 precisely when $\alpha^3 = 1$ and $\alpha \neq 1$. But observe that
\[
	\bp \alpha & \\ & 1 \ep 
	\bp & 1 \\ 1 & \ep 
	\bp \alpha^{-1} & \\ & 1\ep =
	\bp & \alpha^2 \\ 1 & \ep \not\in N,
\]
so that the full subgroup $N$ is not stable under conjugation by $s$. Hence
\eqref{Eq: Case 1} may discarded.

Let us now suppose that \eqref{Eq: Case 2} holds, so that $\beta \gamma = -
\alpha \delta$. By Lemma~\ref{Lem: order n}, if $s$ has order 3, then
\[
0 = \tr(s)^2 - \det(s) = (\alpha + \delta)^2 - \alpha \delta + \beta \gamma = \alpha^2 + \delta^2.
\]
Hence $\delta = \pm i \alpha$. Squaring both sides of the second equation in
\eqref{Eq: Case 2} and dividing by $\alpha^2 = -\delta^2$, we see that $\beta^2
= - \gamma^2$. Without loss of generality, we may suppose that $\gamma = 1$, so
that $\beta = \pm i$. Squaring both sides of the first equation in \eqref{Eq:
  Case 2} and replacing $\delta^2$ with $- \alpha^2$ and $\beta^2$ with
$-\gamma^2 = -1$, we find that $\alpha^4 = 1$. We conclude that $s =
\mat{\varepsilon_1 & \beta \\ 1 & \varepsilon_2}$, where $\varepsilon_j^4 = 1$
for $j = 1, 2$ and $\beta = \pm i$. In order for the first equation of
\eqref{Eq: Case 2} to be satisfied, we must have $\beta = - \varepsilon_1
\varepsilon _ 2$, so that exactly one of $\varepsilon_1$ and $\varepsilon_2$ is
a primitive fourth root of unity, while the other is $\pm 1$. Hence the elements
of order 3 that normalize $N$ lie in the following set:
\[
\left\{ \bp \varepsilon  & - \varepsilon \varepsilon ' i \\ 1 & \varepsilon' i \ep 
	 \ : \ \varepsilon^2 = (\varepsilon')^2 = 1 \right\} \cup
\left\{ \bp \varepsilon i & - \varepsilon \varepsilon ' i \\ 1 & \varepsilon' \ep
	 \ : \ \varepsilon^2 = (\varepsilon')^2 = 1 \right\}.
\]
As there are 8 elements in this set, and since a tetrahedral group has 8
elements of order 3, we have found all of them.

Let $s \in G$ be any element of order~3. Evidently $N \cup Ns \cup Ns^2 = G$, so
that $G = N \rtimes \langle s \rangle$. Now choose $\varepsilon = 1$ and $\varepsilon'
= -1$ in the first of the above sets of elements of order 3 to arrive at the
desired generator~$G$.
\end{proof}

\begin{corollary}
  \label{cor: tetrahedral}
If $p = 3$ and $G \subset \PGL_2(k)$ is tetrahedral, then $G$ is conjugate to
$\PSL_2(\FF_3)$.
\end{corollary}

\begin{proof}
The group $\PSL_2(\FF_3)$ acts faithfully on the set $\PP^1(\FF_3)$, which has
four points. This gives an injective homomorphism $\PSL_2(\FF_3) \to
\mathfrak{S}_4$. By comparing orders, we see that the image in $\mathfrak{S}_4$
has index~2, which means $\PSL_2(\FF_3) \cong \mathfrak{A}_4$. Thus
$\PSL_2(\FF_3)$ is tetrahedral, and the preceding proposition shows $G$ and
$\PSL_2(\FF_3)$ must be conjugate.
\end{proof}


\subsection{Octahedral subgroups}
\label{Sec: Octahedral}

Recall that the group of orientation-preserving symmetries of a regular
octahedron is isomorphic to $\mathfrak{S}_4$. (Look at the action of the
symmetry group on the set of pairs of opposite faces, of which there are four.)
Any group isomorphic to $\mathfrak{S}_4$ will therefore be called
\textbf{octahedral}.

\begin{lemma}
\label{Lem: Octahedral}
Let $G$ be a group of order~$24$ such that (a) $G$ has no central element of
order~$2$, and (b) $G$ has exactly~$4$ conjugate cyclic subgroups of
order~$3$, each of which has normalizer equal to a dihedral subgroup of
order~$6$. Then $G$ is octahedral.
\end{lemma} 

\begin{proof}
Let $C_1, \ldots, C_4$ be the four conjugate cyclic subgroups of order~3, and
let $D_1, \ldots, D_4$ be the associated dihedral normalizers. Note that the
$D_i$ must also be conjugate. We claim that $D_1 \cap \cdots \cap D_4 =
\{e\}$. If it contains an element $g$ of order 3, then $g$ belongs to each of
the $C_i$, and hence the $C_i$ are not distinct, a contradiction. If the
intersection contains a pair of distinct elements of order~2, then it contains
their product, which has order~3, another contradiction. If the intersection
contains a unique element of order~2, say $g$, then $sgs^{-1}$ lies in the
intersection as well for every $s \in G$. Hence $sgs^{-1} = g$, or $g$ lies in
the center of $G$, a final contradiction. Thus $D_1 \cap \cdots \cap D_4$ is
trivial.

Consider the action of $G$ on the set $\{C_1, \ldots, C_4\}$ given by
conjugation; it induces a homomorphism $\phi: G \to \mathfrak{S}_4$. Suppose
that $g$ lies in the kernel of $\phi$. Then $gC_ig^{-1} = C_i$ for each $i$, so
that $g$ belongs to each normalizer $D_i$. We showed above that the intersection
of the normalizers is trivial, so $g = e$. We deduce that $\phi$ is injective,
and that $G \cong \mathfrak{S}_4$.
\end{proof}

\begin{proposition}
Suppose $p \ne 2$ and $G \subset \PGL_2(k)$ is an octahedral subgroup. Then up
to conjugation, $G$ is generated by the tetrahedral subgroup $T = N \rtimes C$
given by Proposition~\ref{Prop: Tetrahedral p} and the element $\mat{i & \\ &
  1}$, where $i$ is any primitive fourth root of unity. In particular, any two
octahedral subgroups of $\PGL_2(k)$ are conjugate when $\mathrm{char}(k)$ is
different from~2.
\end{proposition}

\begin{remark}
\label{Rem: No Octahedral}
When $p = 2$, we know that every element of finite order in $\PGL_2(k)$ has
order~2 or odd order (Proposition~\ref{Prop: Cyclic}). It follows that
$\PGL_2(k)$ does not contain an octahedral subgroup since such groups have
elements of order~4.
\end{remark}

\begin{proof}
Evidently $G$ contains a tetrahedral subgroup $T$, so after conjugation, we may
assume $T = N \rtimes C$ as in Proposition~\ref{Prop: Tetrahedral p}. Since $[G
  : T] = 2$, we know $G$ is generated by $T$ and any element $s \in G$ of
order~4. Then $s^2$ has order~2 and corresponds to an even permutation in
$\mathfrak{S}_4$, so that it lies in $N$. The three nontrivial elements of $N$
are conjugate via elements of $C$, so we may assume that $s^2 = \mat{-1 & \\ &
  1}$. Now $s$ fixes two points of $\PP^1(k)$, and its square fixes the same two
points. Hence $s = \mat{\varepsilon & \\ & 1}$ for some $\varepsilon \in
k^\times$. Since $s$ has order~4, we must have $\varepsilon = \pm i$. Replacing
$s$ with $s^{-1}$ if necessary, we find that $\varepsilon = i$.
\end{proof}

\begin{corollary}
  \label{cor:octahedral3}
If $p = 3$ and $G \subset \PGL_2(k)$ is octahedral, then $G$ is conjugate to
$\PGL_2(\FF_3)$.
\end{corollary}

\begin{proof}
The group $\PGL_2(\FF_3)$ acts faithfully on the set $\PP^1(\FF_3)$, which has
four points. This gives an injective homomorphism $\PGL_2(\FF_3) \to
\mathfrak{S}_4$. As these groups have the same order, they must be
isomorphic. Thus $\PGL_2(\FF_3)$ is octahedral, and the preceding proposition
shows $G$ and $\PGL_2(\FF_3)$ must be conjugate.
\end{proof}


\subsection{Icosahedral subgroups}
\label{Sec: Icosahedron}

Recall that the group of orientation-preserving symmetries of a regular
icosahedron is isomorphic to $\mathfrak{A}_5$. (See
\cite[\S3$\cdot$6--3$\cdot$7]{Coxeter_Regular_Polytopes}.)
Any group isomorphic to $\mathfrak{A}_5$ will therefore be called \textbf{icosahedral}.

\begin{lemma}
\label{Lem: Icosahedral}
Let $G$ be a group of order~$60$ with exactly ten conjugate $3$-subgroups and
exactly fifteen elements of order~$2$ lying in five conjugate Klein
4-groups. Then $G$ is icosahedral.
\end{lemma}

\begin{proof}
Let $K_1, \ldots, K_5$ be the conjugate Klein 4-subgroups. We let $G$ act on the
set $\{K_1, \ldots, K_5\}$ by conjugation, so that we have a homomorphism $\phi:
G \to \mathfrak{S}_5$. If we can show that $G$ is injective, then it is isomorphic to an
index~2 subgroup of $\mathfrak{S}_5$, which must be $\mathfrak{A}_5$.
	
First note that if $N_i$ is the normalizer of $K_i$ in $G$, then $|N_i| = 12$ by
the orbit-stabilizer theorem. Each $N_i$ is tetrahedral by Lemma~\ref{Lem:
  Tetrahedral}. Indeed, it suffices to show that $N_i$ is non-abelian. But if it
were abelian, then it would contain a normal subgroup $C_i$ of order~3, which
would be one of at most five conjugate Sylow 3-subgroups of $G$. But $G$ has ten
conjugate 3-subgroups, a contradiction.
	
To show that $\phi$ is injective, we must prove that $N_1 \cap \cdots \cap N_5 =
\{e\}$.  Write $N$ for this intersection. Then $N$ is a normal subgroup of $G$,
and hence of each $N_i$. Now $N_i$ is tetrahedral, so its only normal subgroups
are its trivial subgroups and $K_i$. No two of the $N_i$ are equal since they
contain conjugate subgroups $K_i$; hence $N \neq N_i$ for any $i$. The $K_i$
have only the identity in common as they contain all fifteen of the elements of
$G$ of order 2. We conclude that $N \neq K_i$ for any $i$. So $N = \{e\}$.
\end{proof}

\begin{lemma}
\label{Lem: Icosahedral presentation}
An icosahedral group $G$ can be generated by two elements $g,h$ subject to the
relations $g^5 = h^2 = (gh)^3 = 1$.
\end{lemma}

\begin{proof}
We may assume $G = \mathfrak{A}_5$. Let $g = (12345)$ and $h = (12)(34)$. Then
$gh = (135)$ has order 3. Let $H = \langle g, h \rangle \subset G$. Evidently
$g$ and $h$ have the correct relations, so it suffices to prove that $|H| =
|G|$. By Cauchy, $H$ contains subgroups of order 3 and 5. We now show that $H$
has a subgroup of order 4, so that $|H|$ is divisible by $4 \cdot 3 \cdot 5 =
60$. We have the following relations:
\begin{eqnarray*}
	g^{-1}hg &=& (15)(23) \\
	  (ghg^{-1}) h (ghg^{-1}) &=&  (13)(25)
\end{eqnarray*}
These two products of 2-cycles generate a Klein 4-subgroup of $H$, which
completes the proof.
\end{proof}

\begin{proposition}
\label{Prop: Icosahedral conjugacy}
Suppose $p \neq 5$ and $G \subset \PGL_2(k)$ is icosahedral. For any primitive
fifth root of unity $\zeta$, the group $G$ is conjugate to the group $\langle s,
t \rangle$, where $s = \mat{\zeta & \\ & 1}$ and $t = \mat{1 & \ \ 1 - \zeta -
  \zeta^{-1} \\ 1 & -1}$.  These generators satisfy $s^5 = t^2 = (st)^3 = I$. In
particular, any two icosahedral subgroups of $\PGL_2(k)$ are conjugate when
$\mathrm{char}(k) \neq 5$.
\end{proposition}

\begin{proof}
We begin by showing that $G$ is conjugate to a subgroup of the sort given in the
proposition for \textit{some} primitive fifth root of unity $\zeta$; afterward,
we will show that we may specify $\zeta$. Let $s, t \in G$ be generators as in
Lemma~\ref{Lem: Icosahedral presentation}; i.e., $s$ has order~5, $t$ has
order~2, and $(st)^3 = I$. Since $p \neq 5$, $s$ must fix two elements of
$\PP^1(k)$. After conjugation, we may assume that $s = \mat{ \zeta & \\ & 1}$
with $\zeta$ some primitive fifth root of unity.
	
Since $t$ has order~2, it may be written as $t = \mat{\alpha & \beta \\ \gamma &
  -\alpha}$ (Lemma~\ref{Lem: order n}). Now
\[
st = \bp \zeta & \\ & 1 \ep
	\bp \alpha & \beta \\ \gamma & -\alpha \ep
	= \bp \zeta \alpha & \zeta \beta \\ \gamma & - \alpha \ep.
\]
The condition for $st$ to have order 3 is
\[
0 = \tr(st)^2 - \det(st) = \zeta^2 \alpha^2  + \zeta (\gamma \beta - \alpha^2) + \alpha^2.
\]
If $\alpha = 0$, then this implies $\det(t) = 0$. So we may assume that $\alpha
= 1$. Now the previous equation becomes
\begin{equation}
\label{Eq: beta gamma}
\beta \gamma = - \frac{1}{\zeta}(\zeta^2 - \zeta + 1) 
	= - \frac{\zeta^3 + 1}{\zeta(\zeta + 1)}.
\end{equation}
As $\zeta$ is a primitive fifth root of unity, we find $\beta \gamma \neq 0$. If
we conjugate $G$ by $\mat{\gamma & \\ & 1}$, then the subgroup generated by
$\mat{\zeta & \\ & 1}$ is unaffected, while
\[
\bp \gamma &  \\ & 1 \ep
\bp 1 & \beta \\ \gamma & -1 \ep
\bp \gamma^{-1} & \\ & 1 \ep = 
\bp 1 & \gamma \beta \\ 1 & -1 \ep.
\]
So without loss of generality, we may assume that $\gamma = 1$. From \eqref{Eq:
  beta gamma}, we find that
\[
t =  \bp 1 & - (\zeta^2 - \zeta + 1) / \zeta \\  1 & -1 \ep
= \bp  1  & \ \ 1 - \zeta - \zeta^{-1}  \\  1 & -1\ep.
\]
By construction, $s^5 = t^2 = (st)^3 = I$.

It remains to show that different fifth roots of unity give rise to conjugate
subgroups of $\PGL_2(k)$. For each $i = 1, 2, 3, 4$, let
\[
	s_i = \bp \zeta^i & \\ & 1\ep, \quad t_i = \bp 1 & 1 - \zeta^i - \zeta^{-i} \\ 1 & -1 \ep, 
	\quad G_i = \langle s_i, t_i \rangle.
\]
Evidently the symmetry $i \mapsto -i$ in $t_i$ shows $G_1 = G_4$ and $G_2 =
G_3$; in general, there are no further equalities among the $G_i$.
If we let $\lambda = \zeta^3 - \zeta^2 + \zeta$, then a direct calculation shows that
\[
	\bp \lambda & \\ & 1 \ep s_2 \bp \lambda^{-1} & \\ & 1 \ep = s_1^2 \in G_1 \qquad  \text{and} \qquad
	\bp \lambda & \\ & 1 \ep t_2 \bp \lambda^{-1} & \\ & 1 \ep = t_1s_1t_1s_1^{-1}t_1 \in G_1.
\]
It follows that $\Big(\begin{smallmatrix}\lambda & \\ & 1 \end{smallmatrix} \Big)
G_2 \Big(\begin{smallmatrix} \lambda^{-1} & \\ & 1 \end{smallmatrix} \Big) \subset
G_1$, and since $G_1$ and $G_2$ have the same order, we have proved they are
conjugate.
\end{proof}

\begin{proposition}
  \label{prop:icosahedral5}
Suppose $p = 5$ and $G \subset \PGL_2(k)$ is icosahedral. Then $G$ is conjugate
to $\PSL_2(\FF_5)$. Moreover, we can take $s = \mat{1 & 1 \\ & 1}$ and $t =
\mat{ & -1 \\ 1 &}$ as generators such that $s^5 = t^2 = (st)^3 = I$.
\end{proposition}

\begin{proof}
The strategy is essentially the same as in the previous proposition. First we
choose elements $s,t$ with $s^5 = t^2 = (st)^3 = I$ (Lemma~\ref{Lem: Icosahedral
  presentation}). Since $p = 5$, $s$ is unipotent, so we may conjugate to get $s
= \mat{1 & 1 \\ & 1}$. Write $t = \mat{ \alpha &
  \beta\\ \gamma & - \alpha}$.  Then
\[
	st = \bp 1 & 1 \\ & 1 \ep 
	\bp \alpha & \beta \\ \gamma & - \alpha \ep = 
	\bp \alpha + \gamma & \beta - \alpha \\ \gamma & - \alpha \ep.
\]
The condition for $st$ to have order~3 is
\[
	\tr(st)^2 - \det(st) = \alpha^2 + \beta \gamma + \gamma^2 = 0.
\]
If $\gamma = 0$, then this implies $\alpha = 0$, so that $\det(t) = 0$. Hence
$\gamma \neq 0$, and we may as well assume that $\gamma = 1$. The previous
equation then implies $\beta = - \alpha^2 - 1$. That is, $t = \mat{ \alpha & -
  \alpha^2 - 1 \\ 1 & - \alpha}$. Finally, we conjugate $G$ by $\mat{1 & -
  \alpha \\ & 1}$. This does not affect the subgroup generated by $s$, but it
does give
\[
	\bp 1 & - \alpha \\ & 1 \ep
	\bp \alpha & - \alpha^2 - 1 \\ 1 & - \alpha \ep 
	\bp 1 & \alpha \\ & 1 \ep  =
	\bp & -1 \\ 1 & \ep.
\]
Hence we may assume without loss of generality that $t = \mat{ & -1 \\ 1 &}$.

We have shown that, up to $\PGL_2(k)$-conjugacy, we have $G = \langle \mat{1 & 1
  \\ & 1}, \mat{ & -1 \\ 1 &} \rangle \subset \PSL_2(\FF_5)$. But these two
groups have the same order, so that $G = \PSL_2(\FF_5)$.
\end{proof}


\section{The \texorpdfstring{$p$}-regular case}
\label{Sec: p-regular}

\begin{convention*}
Throughout this section, we will assume $k$ is an algebraically closed field.
\end{convention*}

Our goal for this section is to prove Theorem~C \textit{when $k$ is an
  algebraically closed field}.

Let $G \subset \PGL_2(k)$ be a finite $p$-regular subgroup. Any nontrivial
element $s \in G$ fixes a unique pair of points $\{x_s, y_s\}$, and so there is
a maximal cyclic subgroup $G(s) \subset G$ containing $s$, namely the set of all
elements of $G$ fixing $x_s$ and $y_s$. Let $N(s)$ be its normalizer in $G$;
then $[N(s) : G(s)] = 1$ or $2$ by Proposition~\ref{Prop: Dihedral
  normalizer}. By letting $G$ act by conjugation on its maximal cyclic
subgroups, we find that $G(s)$ lies in a system of $|G| / |N(s)|$ conjugate
subgroups. Let $G_1, \ldots, G_r$ be a complete set of representatives of the
conjugacy classes of maximal cyclic subgroups of $G$. Let $d_i = |G_i| \geq 2$
and $f_i = [N_G(G_i) : G_i]$. As $G$ is $p$-regular, each of its nontrivial
elements lies in a unique maximal cyclic subgroup. This yields
\[
	|G| = 1 + \sum_{i = 1}^r (d_i - 1) \frac{|G|}{d_if_i}.
\]
Dividing by $|G|$ and rearranging, we have
\begin{equation}
\label{Eq: 1 / |G|}
\frac{1}{|G|} = 1 - \sum_{i = 1}^r \frac{1}{f_i} \left( 1 - \frac{1}{d_i} \right), \text{ and } 
\end{equation}
\begin{equation}
\label{Eq: d_if_i bound}
d_i f_i \leq |G| \quad (i = 1, \ldots, r).
\end{equation}
The summands on the right side of \eqref{Eq: 1 / |G|} have size at least
$\frac{1}{2}\left(1 - \frac{1}{2}\right) = \frac{1}{4}$; as the left side is
positive, we find that $r \leq 3$. In the remainder of the proof, we treat the
various cases that can occur for $r, f_i, d_i$.

\bigskip

\textbf{Case $r = 1$.}  If $f = 2$, then \eqref{Eq: 1 / |G|} implies $|G| = 2d /
(d+1)$, which is not an integer. Hence $f = 1$, and \eqref{Eq: 1 / |G|} gives
$|G| = d$. That is, $G$ is cyclic.

\bigskip

\textbf{Case $r = 2$.} In this case, \eqref{Eq: 1 / |G|} becomes
\[
1 - \frac{1}{|G|} = \frac{1}{f_1} \left( 1 - \frac{1}{d_1} \right) 
	+\frac{1}{f_2} \left( 1 - \frac{1}{d_2} \right)
\]
If $f_1 = f_2 = 1$, then the left side is smaller than~1 while the right side is
$\geq 1$. If $f_1 = f_2 = 2$, then \eqref{Eq: 1 / |G|} and \eqref{Eq: d_if_i
  bound} become
\[
\frac{2}{|G|} = \frac{1}{d_1} + \frac{1}{d_2}, \quad \frac{2}{|G|} \leq \frac{1}{d_i}.
\]
Evidently this is impossible, so we may assume without loss of generality that
$f_1 = 1$ and $f_2 = 2$.

Now we find that 
\[
\frac{1}{|G|} = \frac{1}{d_1} + \frac{1}{2d_2} - \frac{1}{2} \leq \frac{1}{d_1} - \frac{1}{4}, 
\]
so that $d_1 =2$ or $3$. If $d_1 = 2$, then $|G| = 2d_2$, so that $G$ is
dihedral. If $d_1 = 3$, then $1 / |G| = 1 / (2d_2) - 1/6$, so that $d_2 =
2$. Thus $|G| = 12$. We claim $G$ is tetrahedral. The normalizer of a Sylow
3-subgroup is not all of $G$ since $f_1 = 2$; in particular, $G$ is
non-abelian. The subgroups of order~2, of which there are $|G| / 2f_2 = 3$, form
a single conjugacy class, so that they generate a normal subgroup of
order~4. Hence $G$ is tetrahedral by Lemma~\ref{Lem: Tetrahedral}.

\bigskip

\textbf{Case $r = 3$.} Here we must have $f_1 = f_2 = f_3 = 2$. Indeed, if $f_1
= 1$, then \eqref{Eq: 1 / |G|} becomes
\[
\frac{1}{|G|} = \frac{1}{d_1} - \frac{d_2 - 1}{f_2d_2} - \frac{d_3 - 1}{f_3d_3} \leq \frac{1}{d_1}
 	-\frac{1}{4} - \frac{1}{4} \leq 0,
\]
an evident contradiction. So letting $f_i = 2$ for all $i$, \eqref{Eq: 1 / |G|}
is equivalent to
\[
	1 + \frac{2}{|G|} = \frac{1}{d_1} + \frac{1}{d_2} + \frac{1}{d_3}.
\]
If every $d_i \geq 3$, then the right side is at most~1 while the left is
strictly larger than 1. So without loss of generality, we have $d_3 = 2$:
\[
\frac{1}{2} + \frac{2}{|G|} = \frac{1}{d_1} + \frac{1}{d_2}.
\]
If $d_1$ or $d_2$ is~2, we may take $d_2 = 2$, so that $|G| = 2d_1$ and $G$ is
dihedral. Otherwise, we have $d_1 > 2$ and $d_2 > 2$. The above equation implies
that both $d_1$ and $d_2$ cannot be larger than~3. So let us suppose $d_2 =
3$. Thus
\[
\frac{1}{6} + \frac{2}{|G|} = \frac{1}{d_1}.
\]
Hence $d_1 < 6$. For $d_1 = 3, 4, 5$, we find $|G| = 12, 24, 60$,
respectively. We treat these cases separately now.

If $d_1 = 3, d_2 = 3, d_3 = 2$, we find that $G$ has two non-conjugate subgroups
of order~3. But $G$ has order $12 = 3\cdot4$, so we have contradicted the Sylow
theorems.
	
If $d_1 = 4, d_2 = 3, d_3 = 2$, then $|G| = 24$, and $G$ is octahedral by
Lemma~\ref{Lem: Octahedral}. Indeed, to check that $G$ has no central element of
order~2, observe that if $s$ were such an element, then it would fix exactly two
points $x$ and $y$. By commutativity, the subgroups of order~3 would act on
these two points, hence fixing them, and hence $s$ lies in a maximal cyclic
subgroup containing an element of order~3, a contradiction.
	
If $d_1 = 5, d_2 = 3, d_3 = 2$, then we argue that $G$ is of icosahedral
type. Let $n_i$ be the number of elements of $G$ of order $i$. By hypothesis on
the $d_j$'s, we see that $n_i = 0$ if $ i \ne 1, 2, 3, 5$. Thus, $n_1 = 1$ and
\begin{equation}
  \label{eq:elt_count}
  60 = 1 + n_2 + n_3 + n_5. 
\end{equation}
The Sylow 3- and 5-subgroups are not normal in $G$ because $f_1 = f_2 = 2$. By
the Sylow theorems, there are 6 Sylow 5-subgroups, either 4 or 10 Sylow
3-subgroups, and 1,3,5, or 15 Sylow 2-subgroups. Note that each Sylow 2-subgroup
is dihedral of order 4 --- i.e., a Klein 4-group. Thus, $n_5 = 24$, $n_3 = 8$ or
$20$, and $n_2 = 3$, $9$, $15$, or $45$. To satisfy \eqref{eq:elt_count}, we
must have $n_3 = 20$ and $n_2 = 15$. That is, there are 10 conjugate 3-subgroups
and 15 elements of order 2 lying in 5 conjugate Klein 4-groups. Apply
Lemma~\ref{Lem: Icosahedral} to see that $G$ is icosahedral.

We have now shown that the groups presented in Theorem~C constitute all possible
isomorphism classes of finite $p$-regular subgroups of $\PGL_2(k)$. In the
previous section we constructed all of these groups and showed that they are
unique up to $\PGL_2(k)$-conjugation.

\begin{remark}
  For the reader with a background in algebraic geometry, Klein originally
  deduced equation \eqref{Eq: 1 / |G|} from the Riemann-Hurwitz formula for the
  quotient map $\PP^1 \to \PP^1 / G$; see \cite[Ch.V.2]{Klein_Icosahedron}.
\end{remark}

        
\section{Subgroups with elements of order~\texorpdfstring{$p$}\phantom{ }}
\label{Sec: p-irregular}

\begin{convention*}
 Throughout this section we will assume $k$ is an algebraically closed field.
\end{convention*}

This entire section is devoted to a proof of Theorem~B \textit{in the case where
  $k$ is an algebraically closed field}. More precisely, the statement reduces
to the following:

\begin{theorem}
\label{Thm: Finite subgroups all p, algebraically closed}
Let $k$ be an algebraically closed field of characteristic $p > 0$.
\begin{enumerate}
\item Let $q$ be a power of $p$.  There is exactly one conjugacy class of
  subgroups of $\PGL_2(k)$ isomorphic to each of $\PSL_2(\FF_q)$ and
  $\PGL_2(\FF_q)$.
					
\item Let $n \in \NN \smallsetminus p\NN$ and $m \in \NN$ with $m \ge 1$.  The
  conjugacy classes of $p$-semi-elementary subgroups of $\PGL_2(k)$ of order
  $p^m n$ are parameterized by the set of homothety classes of rank-$m$
  subgroups $\Gamma \subset k$ that are stable under multiplication by elements
  of $\mu_n(k)$. The correspondence is
\[
\Gamma \mapsto \bp \mu_n(k) & \Gamma \\ & 1 \ep.
\]

\item Suppose that $p = 2$ and $n > 1$ is an odd integer.  Then there is a
  unique conjugacy class of dihedral subgroups of $\PGL_2(k)$ of order~$2n$.
					  
\item If $p = 3$, then there is exactly one conjugacy class of subgroups of
  $\PGL_2(k)$ isomorphic to $\mathfrak{A}_5$.
\end{enumerate}
Any $p$-irregular subgroup of $\PGL_2(k)$ is among the four types listed here.
\end{theorem}	

\begin{remark}
Evidently the above result gives Theorem~B when $k$ is algebraically closed
except perhaps when $p = 2$ and $G$ is dihedral. But in this latter case, we
observe that $(k^\times)^2 = k^\times$, and so the two statements agree.
\end{remark}

Let us begin the proof of Theorem~\ref{Thm: Finite subgroups all p,
  algebraically closed}. Suppose that $G \subset \PGL_2(k)$ is a finite subgroup
containing an element of order~$p$. Write $|G| = p^m n$ with $p \nmid n$ and $m
\geq 1$, and fix a Sylow $p$-subgroup $P \subset G$. Without loss of generality,
we may conjugate $G$ so that $P = \mat{1 & \Gamma \\ & 1}$ for some additive
subgroup $\Gamma \subset k$ of rank~$m$ (\S\ref{Sec: p-groups}). Let $N =
N_G(P)$ be the normalizer of $P$ in $G$. Every element of $N$ fixes $\infty \in
\PP^1(k)$. Indeed, every element of $P$ fixes $\infty$, and
\begin{align*}
s \in N \ &\Longrightarrow \ s u s^{-1} \in P \quad (u \in P) \\
  &\Longrightarrow \ s u s^{-1}.\infty = \infty \quad (u \in P) \\
  & \Longrightarrow \ u.(s^{-1}.\infty) = s^{-1}.\infty \quad (u \in P) \\
   & \Longrightarrow s^{-1}.\infty = \infty.
\end{align*}
After a suitable conjugation of $G$ we may suppose that $N \subset P \rtimes
\mat{\mu_d(k) & \\ & 1}$ for some integer $d$ coprime to $p$ with $ \mu_d(k)
\subset \FF_\Gamma^\times$ (Proposition~\ref{Prop: Borel Normalization}).  But
$P$ is normal is this semi-direct product, so $N = P \rtimes \mat{\mu_d(k) &
  \\ & 1}$. Let us write $\FF_\Gamma = \FF_{p^\ell}$. As $\Gamma$ is an
$\FF_\Gamma$-vector space, we find $\ell \mid m$.

If $P$ is normal in $G$, then $N = G$ is a subset of the standard Borel subgroup; we
have already dealt with this case in \S\ref{Sec: Borel}. Now suppose that $P$ is
not normal in $G$, and let us count the elements of $G$ of order $p$ in two
different ways.

First, let $P$ act on $G$ by conjugation. If $Q$ is another Sylow $p$-subgroup,
then $Q$ fixes a unique point $x \in \PP^1(k) \smallsetminus \{\infty\}$. As $s$
varies over $P$, we find $s.x$ varies over a set of $|P| = p^m$ distinct
elements.  For $s \in P$, we also have $sQs^{-1}.(s.x) = s.x$, so that
$sQs^{-1}$ fixes $s.x$. Thus, the orbit of $Q$ under the conjugation action of
$P$ has cardinality $p^m$. Writing $f>0$ for the number of orbits of Sylow
$p$-subgroups distinct from $P$, it follows that
\begin{equation}
\label{Eq: Order p guys}
\left| \left\{s \in G \ : \ s^p = I, s \ne I \right\} \right| = (|P| -1) + fp^m (|P| -1) = (1 + fp^m)(p^m - 1),
\end{equation}
and the elements of order $p$ lie in $1 + fp^m$ Sylow $p$-subgroups of $G$. Note
further that $G$ acts transitively on the set of Sylow $p$-subgroups by
conjugation, and the stabilizer of $P$ under this action is precisely $N$. Hence
\begin{equation*}
	|G| = |N| \cdot |1 + fp^m| = (1 + fp^m)p^md. 
\end{equation*}

Second, we estimate the number of elements of order~$p$ in $G$ as follows. Let
$\{s_i : i = 1, \ldots, fp^m\}$ be representatives of the nontrivial cosets of $
G / N$, and write $s_i = \mat{ \alpha_i & \beta_i \\ \gamma_i & \delta_i}$. Note
that $\gamma_i \neq 0$, else $s_i \in N$. If $t_{\lambda, \mu} = \mat{\lambda &
  \mu \\ & 1} \in N$, then we have
\[
s_i t_{\lambda, \mu} = 
	\bp \alpha_i \lambda & \alpha_i \mu + \beta_i \\ \gamma_i \lambda & \gamma_i \mu + \delta_i \ep.
\] 
By Lemma~\ref{Lem: order p}, $s_i t_{\lambda, \mu}$ has order $p$ if and only if
\begin{equation}
\label{Eq: order $p$ rep}
\left( \alpha_i \lambda + \gamma_i \mu + \delta_i\right)^2 = 4\lambda \det(s_i).
\end{equation}
Define $\varepsilon_2 = 1$ and $\varepsilon_p = 2$ for $p \ge 3$ prime. For
fixed $s_i$ and $\lambda$, there are precisely $\varepsilon_p$ values of $\mu
\in k$ such that \eqref{Eq: order $p$ rep} is satisfied. So for a given $s_i$,
there are at most $\varepsilon_p d$ elements $t_{\lambda, \mu} \in N$ such that
\eqref{Eq: order $p$ rep} is satisfied. Combining this argument with \eqref{Eq:
  Order p guys}, we find that
\[
(1 + fp^m)(p^m - 1) = \left| \left\{s \in G \ : \ s^p = I, s \ne I \right\} \right| \leq (p^m - 1) + \varepsilon_p d fp^m.
\]		
Subtracting $p^m - 1$ from both sides yields
\begin{equation}
\label{Eq: sandwich}
\begin{aligned}
	fp^m(p^m - 1) \leq \varepsilon_pdfp^m \ 
	&\Rightarrow \ p^m - 1 \leq \varepsilon_p d \leq \varepsilon_p (p^\ell - 1)  
	\qquad (p^\ell = |\FF_\Gamma|) \\
	&\Rightarrow p^m - 1 < 2p^\ell - 1.
\end{aligned}
\end{equation}
If $\ell < m$, then this gives $p^{m - \ell} < 2$, which is impossible. Since
$\ell \mid m$, we must have $\ell = m$.

For simplicity in what follows, let us write $q = p^m$. Now $\FF_\Gamma =
\FF_q$, $|\Gamma| = q$, and $\mu_d(k) \subset \FF_\Gamma^\times =
\FF_q^\times$. After conjugating by $\mat{\gamma & \\ & 1}$ for some $\gamma \in
\Gamma \smallsetminus\{0\}$, we may assume that $1 \in \Gamma$. Since $\Gamma$
is stable under multiplication by $\FF_\Gamma$, it follows that $\Gamma =
\FF_q$.

The first line of \eqref{Eq: sandwich} gives
\[
	 d \geq \frac{q-1}{\varepsilon_p}.
\]
So if $p = 2$, then $\mu_d(k) = \FF_q^\times$, and if $p > 2$, then $\mu_d(k) =
\FF_q^\times$ or $(\FF_q^\times)^2$. We summarize what has been achieved thus
far.
	
\begin{lemma}
  \label{lem:summary}
If $G$ is a finite subgroup of $\PGL_2(k)$ containing an element of order~$p$,
then up to conjugation, exactly one of the following is true:
\begin{itemize}
\item $G \subset B(k)$ (in which case $G$ is $p$-semi-elementary), or
\item $G$ contains the Sylow $p$-subgroup $\mat{1 & \FF_q \\ & 1}$ with
  normalizer $N = \mat{1 & \FF_q \\ & 1} \rtimes \mat{\Lambda & \\ & 1}$, where
  $\Lambda = \FF_q^\times$ or $\Lambda = (\FF_q^\times)^2$. There exists an
  integer $f > 0$ such that
\[
	|G| = |\Lambda| (1 + fq)q.
\]
\end{itemize}
\end{lemma}

We assume in what follows that we are in the second case of
Lemma~\ref{lem:summary}. It will also be convenient to have the following result
at our disposal.
	
\begin{lemma}
\label{Lem: Different infinities}
Write $G = N \sqcup s_1N \sqcup \cdots \sqcup s_{fq}N$ as above. Then
$s_i.\infty \neq s_j.\infty$ whenever $i \neq j$.
\end{lemma}

\begin{proof}
If $s_i.\infty = s_j.\infty$, then $s_i^{-1}s_j.\infty = \infty$. By the
characterization of $N$ as the largest subgroup of $G$ that fixes $\infty$, we
must have $s_i^{-1}s_j \in N$, or equivalently $s_j \in s_i N$.
\end{proof}
	

\subsection{The case \texorpdfstring{$\Lambda = (\FF_q^\times)^2$}\phantom{ }}
\label{Sec: Lambda = squares}
	
Note that this includes the case $q$ even, where $(\FF_q^\times)^2 =
\FF_q^\times$. For $q > 3$, we will show that, perhaps after a further
conjugation, we have $G \subset \PSL_2(\FF_q)$. Then
\[
\frac{q(fq+1)(q-1)}{\varepsilon_p} = |G| \leq |\PSL_2(\FF_q)| = \frac{q(q^2 - 1)}{\varepsilon_p},
\]
so that $f = 1$ and $G = \PSL_2(\FF_q)$. For $q = 2$, we will show that $G \cong
\mathfrak{D}_{1+2f}$, a dihedral group. For $q = 3$, we will show that $G$ is
tetrahedral, and hence conjugate to $\PSL_2(\FF_3)$ by Corollary~\ref{cor:
  tetrahedral}.

We begin by arranging for the nontrivial coset representatives $s_1, \ldots,
s_{qf}$ to have order $p$.  Since $\Lambda = (\FF_q^\times)^2$, the first
inequality of \eqref{Eq: sandwich} is actually an equality, which implies that
for each fixed coset representative $s_i$, and each $\lambda \in \Lambda$, there
are exactly $\varepsilon_p > 0$ elements $\mu \in \FF_q$ satisfying \eqref{Eq:
  order $p$ rep}. In particular, each coset contains an element of order~$p$, so
after choosing new coset representatives, we may assume that each $s_i$ has
order $p$.

We divide the remainder of the proof into four cases: $q = 2^m$ for $m > 1$,
$q > 3$ odd, $q = 2$, and $q = 3$.


\subsubsection{The Case $q = 2^m$ with $m > 1$.}
\label{Sec: q = 2^m, m > 1}
Write our coset representative as $s_i = \mat{\alpha_i & \beta_i \\ \gamma_i &
  \delta_i}$ for $i = 1, \ldots qf$. Replace $s_i$ with $\gamma_i^{-1}s_i$ in
order to assume that $\gamma_i = 1$. (Recall that $\gamma_i = 0$ would imply
$s_i \in N$, contradicting our setup.) If we can show that $\alpha_i, \beta_i,
\delta_i \in \FF_q$, then we will be able to conclude that $G \subset
\PGL_2(\FF_q) = \PSL_2(\FF_q)$.

As $s_i$ has order~2, we see that $\delta_i = \alpha_i$ (Lemma~\ref{Lem: order
  n}).  We saw above that for each coset representative $s_i$ and each $\lambda
\in \Lambda =\FF_q^\times$, there is a choice of $\mu \in \FF_q$ such that
$s_it_{\lambda,\mu}$ has order~2. Take $\lambda \ne 1$ so that $s_i
t_{\lambda,\mu} \ne s_i$. Looking at \eqref{Eq: order $p$ rep}, we find that
\[
 \left( \alpha_i \lambda + \gamma_i \mu + \delta_i\right)^2 = 0.
\]
Since $\delta_i = \alpha_i$ and $\gamma_i = 1$, we conclude that
$\alpha_i = \mu / (\lambda + 1) \in \FF_q$.

It remains to show that $\beta_i \in \FF_q$. If $s_i s_j \in N$, then $s_j \in
s_iN$, so that $i = j$. So for $i \neq j$ there exist $\ell$ and $t_{\lambda,
  \mu} = \mat{\lambda & \mu \\ & 1} \in N$ such that $s_i s_j = s_\ell
t_{\lambda, \mu}$. It follows that
\begin{equation*}
\alpha_\ell = s_\ell.\infty = s_\ell t_{\lambda, \mu}.\infty 
	= s_i s_j.\infty 
	= s_i.\alpha_j = \frac{\alpha_i \alpha_j + \beta_i}{ \alpha_j - \alpha_i} 
	= \frac{\alpha_i \alpha_j}{\alpha_j - \alpha_i} + 
		\beta_i \frac{1}{\alpha_j - \alpha_i}.
\end{equation*}
Note that $\alpha_j - \alpha_i \neq 0$, else $\alpha_i = \alpha_j$, from which
we deduce that $s_i.\infty = s_j.\infty$, in contradiction to Lemma~\ref{Lem:
  Different infinities}. The above computation shows that $\beta_i \in \FF_q$, as desired. 


\subsubsection{The Case $q > 3$ Odd.}
\label{sec:det(s)=1}
Recall that each of our coset representatives $s_i$ has order~$p$. After a
suitable scaling, we may further assume that $\det(s_i) = 1$. By Lemma~\ref{Lem:
  order p}, to say that $s_i$ has order $p$ means that
\[
(\alpha_i + \delta_i)^2 = 4 \det(s_i) = 4.
\]
Hence $\alpha_i + \delta_i = \pm 2$. Replacing $s_i$ with $-s_i$, we may assume
that $\tr(s_i) = \alpha_i + \delta_i = 2$. We will now argue that $\alpha_i,
\beta_i, \gamma_i, \delta_i \in \FF_q$.

Write $\Lambda = \{\eta^2 : \eta \in \FF_q^\times\}$. Then \eqref{Eq: order $p$
  rep} becomes
\begin{equation}
\label{Eq: Order p coset rep}
(\alpha_i \eta^2 + \gamma_i \mu + 2 - \alpha_i)^2 = 4 \eta^2 \Longleftrightarrow
\alpha_i(\eta^2 - 1) + \gamma_i \mu = -2 + 2\eta.
\end{equation}
The right side lies in $\FF_q$, and we know that for each choice of $i$ and
$\eta \in \FF_q^\times$, there is an element $\mu = \mu_{i,\eta} \in \FF_q$
satisfying the above equation.

Setting $\eta = -1$ in \eqref{Eq: Order p coset rep} shows $\gamma_i = -4 /
\mu_{i, -1} \in \FF_q \smallsetminus \{0\}$.

As $q > 3$, we may choose $\eta \neq \pm 1$. Then $\alpha_i = (-2 + 2\eta -
\gamma_i \mu_{i, \eta})(\eta^2 -1 )^{-1} \in \FF_q$, and $\delta_i = 2 -
\alpha_i \in \FF_q$. Since $\det(s_i) = 1$, it follows that $\beta_i \in \FF_q$
as well. Thus $G \subset \PGL_2(\FF_q)$. But $\det(s_i) = 1$ for each $i$ and
$\det(t)$ is a square for each $t \in N$. Hence $G \subset \PSL_2(\FF_q)$ as
desired.


\subsubsection{The Case $q = 2$.}
Setting $s_0 = \mat{1 & 1 \\ & 1}$ and observing that $N = \{I, s_0\}$, our
setup allows us to write $G = \sqcup_{i = 0}^{2f} s_i N$, where each $s_i$ has
order 2. As $|G| = 2(1+2f)$, each Sylow subgroup has order~2. There are at most
$1+2f$ such subgroups. As $s_0, \ldots, s_{2f}$ are distinct elements of
order~2, they must be all of the elements of order~2. Define
\[
t_i = s_i s_0 \qquad (i = 0, \ldots, 2f).
\]
Then $G = \{t_0, \ldots, t_{2f}, s_0, \ldots, s_{2f}\}$, and each $t_i$ has odd
order.
 	
Define the set $H = \{t_i : i = 0, \ldots, 2f\}$. We claim that $H$ is an
abelian subgroup of $G$. First observe that $t_a s_b \not\in H$ for any $a,
b$. Indeed, since the $s_i$ are all conjugate, there exists $u \in G$ such that
$s_0 = u s_b u^{-1}$. Set $t_c = u t_a u^{-1}$. Then
\[
(ut_a u^{-1})(u s_bu^{-1}) = t_c s_0 = s_c  \ \Rightarrow \ t_as_b = u^{-1} s_c u \in G \smallsetminus \{H\}.
\]
Now define $s_{b'} = t_as_b$. Then
\[
t_a t_b s_0 = t_a s_b = s_{b'} = t_{b'} s_0 \ \Rightarrow \ t_a t_b = t_{b'} \in H,
\]
and hence $H$ is closed under multiplication. Next note that
\begin{equation}
\label{Eq: Conjugate inverse}
	s_0 t_a s_0^{-1} = s_0 t_a s_0 = s_0 s_a = (s_a s_0)^{-1} = t_a^{-1}.
\end{equation}
Thus $t_a^{-1} \in H$ since a conjugate of $t_a$ cannot have order 2, and hence
$H$ is closed under inversion. Finally, for $t_a, t_b \in H$, we have shown that
$t_d := t_b^{-1} t_a^{-1} \in H$. It follows that
\begin{align*}
t_at_b = (t_b^{-1}t_a^{-1})^{-1} = t_d^{-1} 
	&= s_0 t_d s_0^{-1} \qquad \text{by } \eqref{Eq: Conjugate inverse} \\
	&= (s_0 t_b^{-1})(t_a^{-1} s_0^{-1}) \\
	&= (t_bs_0) (s_0t_a) \qquad \text{by } \eqref{Eq: Conjugate inverse} \\
	&= t_b t_a,
\end{align*}
so that $H$ is abelian.

Since $H$ is abelian of odd order, we may apply Theorem~C (which we have already
proved in the algebraically closed setting) to conclude that $H$ is cyclic. If
$t$ is a generator, then $G = \langle t, s_0 \rangle$. Now \eqref{Eq: Conjugate
  inverse} gives $s_0 t s_0 = t^{-1}$, which is precisely the relation that
defines a dihedral group. So $G \cong \mathfrak{D}_{1+2f}$, and there is a
unique conjugacy class of such groups by Proposition~\ref{prop:dihedral}. 


\subsubsection{The Case $q = 3$.}
As $q = 3$, every Sylow 3-subgroup of $G$ is cyclic. Moreover, every element
whose order is divisible by 3 has order exactly 3 (Proposition~\ref{Prop:
  Cyclic}), and hence lies in a Sylow 3-subgroup. We may therefore apply the
argument at the beginning of \S\ref{Sec: p-regular}. Let $G_1, \ldots, G_r$ be a
complete set of representatives for conjugacy classes of maximal cyclic
subgroups; without loss, we assume that $|G_1| = 3$. Since $\Lambda =
(\FF_q^\times)^2 = \{1\}$ in our setting, we see that $d_1 = 3$ and $f_1 =
1$. Thus, \eqref{Eq: 1 / |G|} and \eqref{Eq: d_if_i bound} become
\begin{equation}
  \label{eq:q=3formula}
\frac{1}{|G|} = \frac{1}{3} - \sum_{i=2}^r \frac{1}{f_i} \left( 1 - \frac{1}{d_i} \right) \qquad
\text{and} \qquad d_i f_i \le |G| \quad (i=2, \ldots, r). 
\end{equation}
Moreover, since $G_2, \ldots, G_r$ are $p$-regular, we find that $f_i = 1$ or
$2$ for $i \ge 2$.  As in \S\ref{Sec: p-regular}, each summand above is at
least $\frac{1}{4}$, so that $r = 1$ or $2$. If $r = 1$, then $|G| = 3$, and $G$
is contained inside a Borel subgroup of $\PGL_2(k)$. Recall that we are assuming
that $G$ does not lie in a Borel subgroup, so $r = 2$. If $d_2 > 2$, then the
right side of \eqref{eq:q=3formula} is nonpositive. Thus, $d_2 = 2$ and $f_2 =
2$, and $|G| = 12$.

Evidently $G$ is nonabelian since the normalizer of its Sylow $3$-subgroup is
not equal to $G$. The Sylow theorems show that $G$ has 4 Sylow $3$-subgroups,
and hence 8 elements of order~3. Since $d_2 = f_2 = 2$, $G$ contains a dihedral
subgroup $K$ of order~4 --- i.e., a Klein 4-group. As $|G| = 12$, we conclude
there are only 3 elements of order~2. Thus, $K$ is normal, and Lemma~\ref{Lem:
  Tetrahedral} shows that $G$ is tetrahedral. 


\subsection{The Case \texorpdfstring{$\Lambda = \FF_q^\times$, $q$} Odd}

Suppose now that $\Lambda = \FF_q^\times$. Then $|G| = q(q-1)(1+fq)$ with $f >
0$. Observe that $\mat{\Lambda & \\ & 1}$ is the maximal subgroup of $G$ fixing
$0$ and $\infty$. For if $\mat{\alpha & \\ & 1} \in G$, then this element
normalizes $P = \mat{1 & \Gamma \\ & 1}$, and so $\alpha \in \Lambda$.

Let $t_0 = \mat{-1 & \\ & 1} \in G$. We claim that the centralizer of $t_0$ in
$G$ is either $H = \mat{ \FF_q^\times & \\ & 1}$ or the dihedral group $D = H
\rtimes \langle \mat{ & \tau \\ 1 &} \rangle$ for some $\tau \in
k^\times$. Indeed, any element that commutes with $t_0$ must stabilize $0$ and
$\infty$. The maximal subgroup of $G$ that fixes both of these points is $H$. If
the centralizer is larger than $H$, then the remaining elements must swap $0$
and $\infty$;  any such element has the form $\mat{ & \tau \\ 1 &}$. If $\mat{ &
  \tau \\ 1 &}$ and $\mat{ & \tau' \\ 1 &}$ are two such elements, then their
product is $\mat{ \tau/\tau' & \\ & 1} \in H$. Hence, the centralizer is either
$H$ or $D$, as desired.

We now argue that the centralizer of $t_0$ in $G$ cannot be $H$. We are going to
count elements in $G \smallsetminus N$ of order~2 in two different ways to
obtain a contradiction. Letting $G$ act on itself by conjugation, the
orbit-stabilizer theorem shows that there are $|G| / |H| = q(1 + fq)$ elements
conjugate to $t_0$, all of which have order~2. We note that the elements
conjugate to $t_0$ in $N = \Big( \begin{smallmatrix}1 & \FF_q \\ &
  1\end{smallmatrix} \Big) \rtimes \Big( \begin{smallmatrix} \FF_q^\times & \\ &
    1 \end{smallmatrix} \Big)$ must fix $\infty$ and one other $\FF_q$-rational
  point, so that there are $q$ of them in $N$. Hence
\[
\left| \left\{ s \in G \smallsetminus N : s^2 = I \right\}  \right|  \geq q(1+fq) - q = fq^2.
\]
An element of $s_iN$ is of the form
\[
s_i t_{\lambda, \mu} = 
\bp \alpha_i \lambda & \alpha_i \mu + \beta_i \\ \gamma_i \lambda & \gamma_i \mu + \delta_i \ep,
\]
and it has order 2 if and only if $\alpha_i \lambda + \gamma_i \mu + \delta_i =
0$. Note $\gamma_i \ne 0$ else $s_i \in N$. So given $s_i$ and $\lambda$, there
is exactly one $\mu \in \FF_q$ such that $s_i t_{\lambda, \mu}$ has order
2. Combining with the above lower bound for the number of elements of order 2,
we have that
\[
fq^2 \leq \left| \left\{  s \in G \smallsetminus N : s^2 = I \right\}\right| = fq (q-1),
\]
which is absurd. We conclude that $t_0$ is centralized by a dihedral group $D =
\Big( \begin{smallmatrix}\FF_q^\times & \\ & 1\end{smallmatrix} \Big) \rtimes
  \Big\langle \Big( \begin{smallmatrix} & \tau \\ 1& \end{smallmatrix} \Big)
  \Big\rangle$.

Applying the argument in the last paragraph to $D$ instead of $H$, we find that
the number of elements in $G \smallsetminus N$ conjugate to $t_0$ is precisely
$\frac{1}{2}q(1+fq) - q = \frac{1}{2}q(fq - 1)$. Since this is an integer, we
conclude $f$ must be odd. Moreover, it gives the lower bound
\begin{equation}
\label{Eq: Order 2 lower bound}
	\left| \left\{  s \in G \smallsetminus N : s^2 = I \right\} \right| \geq \frac{1}{2}q(fq - 1).
\end{equation}
We are going to count the number of elements of order 2 in $G \smallsetminus N$
in yet another way in order to bound $f$.
	
Since $\gamma_i \neq 0$ for any $i = 1, \ldots, fq$, we may assume that
$\gamma_i = 1$ in what follows. Let $n$ be the number of cosets $s_iN$
containing at least two elements of order~2. For each of these cosets, we may
assume that $s_i$ has order 2, so that $\delta_i = - \alpha_i$. Moreover, for
each such $i$, there exists $(\lambda, \mu) \in \FF_q^\times \times \FF_q
\smallsetminus \{(1,0)\}$ such that $s_i t_{\lambda, \mu}$ has order 2 --- i.e.,
\[
	\tr(s_i t_{\lambda, \mu}) =  \alpha_i(\lambda - 1) + \mu = 0.
\]
Hence $\alpha_i = \mu / (1 - \lambda) \in \FF_q$. For different choices of $i$,
we get different values of $\alpha_i = s_i.\infty$ (Lemma~\ref{Lem: Different
  infinities}), so that $n \leq |\FF_q| = q$. Since $\alpha_i \in \FF_q$, for
fixed $i$ and $\lambda \in \FF_q^\times$ there exists a unique solution $\mu \in
\FF_q$ to the above trace equation.  Thus, a coset $s_iN$ contains precisely
$q-1$ elements of order~2 or one element of order~2.

Let $m$ be the number of cosets $s_iN$ containing exactly one element of
order~2. Note that $m + n \leq fq$, the total number of nontrivial cosets. The
lower bound \eqref{Eq: Order 2 lower bound} for the number of elements of
order~2 combined with the arguments in the last paragraph gives
\begin{equation}
\label{Eq: Bound for f}
	\begin{aligned}
	&\frac{1}{2}q(fq - 1) \leq m + n(q-1) = m + n + n(q-2) \leq fq + q(q-2) \\
	&\qquad \Longrightarrow  f \leq \frac{2q - 3}{q - 2} = 2 + \frac{1}{q-2}.
	\end{aligned}
\end{equation}
Since $f$ is odd, we find $f = 1$, or $f = 3$ and $q = 3$.


\subsubsection{The case $f = 1$.}
Here we have $|G| = |\Lambda| (1 + fq)q = q(q^2 - 1) = |\PGL_2(\FF_q)|$. We now
prove that $G \subset \PGL_2(\FF_q)$, so that this containment must actually be
equality.

We showed above that $G$ contains an element of the form $\mat{ & \tau \\ 1 &}$
with $\tau \in k^\times$.  For each $\mu \in \FF_q$, we have
\[
	v_\mu := \bp 1 & \mu \\ & 1 \ep			
		\bp & \tau \\ 1 & \ep
		\bp 1 & - \mu \\ & 1 \ep = 
		\bp \mu & \tau - \mu^2 \\ 1 & - \mu \ep.
\]
As $\mu$ varies over the elements of $\FF_q$, we get $q$ distinct elements
$v_\mu$ of order~2 satisfying $v_\mu.\infty = \mu$. Since none of them lies in
$N$, each $v_\mu$ must lie in a distinct nontrivial coset $s_iN$
(Lemma~\ref{Lem: Different infinities}). There are only $fq = q$ such cosets, so
we conclude that every coset contains an element of order~2.

We have now shown that $N \subset \PGL_2(\FF_q)$, and that each nontrivial coset
representative may be chosen to have the form $s_i = \mat{ \alpha_i & \beta_i
  \\ 1 & -\alpha_i}$ with $\alpha_i \in \FF_q$. The final paragraph of
\S\ref{Sec: q = 2^m, m > 1} applies verbatim to show that $\beta_i \in \FF_q$ as
well, which proves that $G \subset \PGL_2(\FF_q)$ as desired.


\subsubsection{The case $f = 3, q = 3$.}
Here we find that $|G| = |\Lambda| (1 + fq)q = 60$. Note that all of the
inequalities in \eqref{Eq: Bound for f} becomes equalities in this case, so that
$n = 3$ and $m = 6$, and $G$ contains $12 + 3 = 15$ elements of order~2. We
showed at the beginning of this subsection that $t_0$ is centralized by a
dihedral group $D$ of order~4. Each such dihedral group $D$ contains three of
the elements of order 2, so that $G$ contains five conjugate Klein
4-groups. Moreover, at the beginning of \S\ref{Sec: p-irregular} we showed that
the number of Sylow 3-subgroups is $1 + fq = 10$. It follows that $G$ is
icosahedral (Lemma~\ref{Lem: Icosahedral}). This conjugacy class is unique by
Proposition~\ref{Prop: Icosahedral conjugacy}.



\section{Results of Serre and Beauville}
\label{sec:arithmetic_p_regular}

We now recall an arithmetic criterion for the existence of certain types of
subgroup in $\PGL_2(k)$. Serre proves the analogous statements for $\GL_2(k)$
\cite[\S2]{Serre_proprietes_galoisiennes}, which we adapt for our purposes.

\begin{theorem}
\label{thm:serre_existence} 
  Let $k$ be a field of characteristic $p \ge 0$.
  \begin{enumerate}
    \item If $n$ is coprime to $p$, then $\PGL_2(k)$ contains a cyclic
      subgroup of order $n$ if and only if $\zeta + \zeta^{-1} \in k$ for some
      primitive $n$-th root of unity $\zeta$.
    \item $\PGL_2(k)$ contains a tetrahedral subgroup if and only if
      \begin{itemize}
      \item $p = 2$ and $\FF_4 \subset k$; or
      \item $p \ne 2$ and $-1$ is the sum of two squares in $k$.
      \end{itemize}
    \item $\PGL_2(k)$ contains an octahedral subgroup if and only if $p \ne 2$
      and $-1$ is a sum of two squares in $k$.
    \item $\PGL_2(k)$ contains an icosahedral subgroup if and only if -1 is a
      sum of two squares in $k$ and either 
      \begin{itemize}
      \item $p = 2$ and $\FF_4 \subset k$;
      \item $p = 5$; or 
      \item $p \ne 2, 5$ and $\sqrt{5} \in k$. 
      \end{itemize}
  \end{enumerate}
\end{theorem}

\begin{remark}
  If $p = 2$, then $\PGL_2(k)$ does not contain an element of order~4
  (Proposition~\ref{Prop: Cyclic}). In particular, it cannot contain an
  octahedral subgroup.
\end{remark}

\begin{proof}[Proof of Theorem~\ref{thm:serre_existence}]
  We begin with the statement about cyclic groups. Given a primitive $n$-th root
  of unity $\zeta$ for some $n$, let us set $\lambda = \zeta + \zeta^{-1}$. If
  $\lambda \in k$, then the element $\mat{ \lambda + 1 & -1 \\ 1 & 1} \in
  \PGL_2(k)$ has order~$n$. Indeed, it is conjugate over $k(\zeta)$ to $\mat{
    \zeta & \\ & 1}$ (Corollary~\ref{Cor: Alternate Cyclic}). Conversely,
  suppose $s \in \PGL_2(k)$ has order $n$, and $p \nmid n$. If $s$ fixes two
  points of $\PP^1(k)$, then we may conjugate those points to $0$ and $\infty$
  using an element of $\PGL_2(k)$. In this way, we may assume without loss of
  generality that $s = \mat{ \zeta & \\ & 1}$ for some primitive $n$-th root of
  unity $\zeta \in k$. If instead the two fixed points of $s$ are quadratic over
  $k$, then after passing to a quadratic extension $k' / k$ we are in the
  previous case and $\zeta \in k'$. The minimal polynomial for $\zeta$ over $k$
  is $z^2 - \lambda z + 1$, so we must have $\lambda \in k$.

  To prove the remaining statements, we observe that Serre gives precisely the
  same criteria for existence of tetrahedral, octahedral, and icosahedral
  subgroups of $\GL_2(k)$ \cite[\S2.5]{Serre_proprietes_galoisiennes}. The
  kernel of the natural homomorphism $\pi \colon \GL_2(k) \to \PGL_2(k)$ is
  $k^\times I$, which is the center of $\GL_2(k)$. In particular, since
  $\mathfrak{A}_4$, $\mathfrak{S}_4$, and $\mathfrak{A}_5$ have trivial center,
  $\pi$ induces an isomorphism on any subgroup isomorphic to one of these. Thus,
  such a subgroup exists in $\GL_2(k)$ if and only if it exists in $\PGL_2(k)$.
\end{proof}

Next, we give Beauville's description of the conjugacy classes of $p$-regular
subgroups of $\PGL_2(k)$:

\begin{theorem}[{\cite[Thm.~4.2]{Beauville_Finite_Subgroups_2010}}]
  \label{thm:beauville_classification}
  Let $k$ be a field of characteristic $p \ge 0$.  
  \begin{enumerate}
    \item $\PGL_2(k)$ contains at most one conjugacy class of $p$-regular
      subgroups isomorphic to each of $\ZZ / n\ZZ$ ($n > 2$), $\mathfrak{A_4}$,
      $\mathfrak{S_4}$, and $\mathfrak{A_5}$.
    \item If $p \ne 2$, the conjugacy classes of subgroups of order~2 of
      $\PGL_2(k)$ are parameterized by $k^\times / (k^\times)^2$. The
      correspondence is given by $\tau \mapsto \mat{ & \tau \\ 1 &}$.
    \item If $p \ne 2$, the homomorphism $\overline{\det} \colon \PGL_2(k) \to
      k^\times / (k^\times)^2$ induces a bijection between
      \begin{itemize}
      \item Conjugacy classes of subgroups of $\PGL_2(k)$ isomorphic to $(\ZZ/2\ZZ)^2$, and
      \item Subgroups $G = \{1,\alpha,\beta,\alpha \beta\} \subset k^\times /
        (k^\times)^2$ of order at most~4 such that the conic $x^2 + \alpha y^2 +
        \beta z^2 = 0$ has a point in $\PP^2(k)$. Here we take $\beta = 1$ if
        $|G| \le 2$ and $\alpha = 1$ if $|G| = 1$.
      \end{itemize}
      The reverse correspondence is given by
      \[
      G \mapsto \left\langle
      \bp \lambda & \alpha \mu \\ \mu & -\lambda \ep, 
      \bp & -\alpha \\ 1 & \ep
      \right\rangle,
      \]
      where $(\lambda, \mu, 1)$ is a point on the conic.\footnote{A smooth conic
      in $\PP^2_k$ either has no rational point or is isomorphic to a line. In
      the latter case, it must have a point with nonzero $z$-coordinate.}
    \item Let $n > 2$, and assume that $k$ contains a primitive $n$-th root of
      unity, say $\zeta$. Let $\mathfrak{Dih}_n(k)$ denote the set of conjugacy
      classes of $p$-regular dihedral subgroups of $\PGL_2(k)$ of order
      $2n$. Then $\mathfrak{Dih}_n(k)$ is parameterized by $k^\times /
      \mu_n(k)(k^\times)^2$, and the correspondence is
      \[
      \alpha \mapsto \left\langle
      \bp \zeta & \\ & 1 \ep,
      \bp & \alpha \\ 1 \ep
      \right\rangle.
      \]
  \end{enumerate}
\end{theorem}

\begin{remark}
In the version of Theorem~\ref{thm:beauville_classification}(3) that appears in
\cite{Beauville_Finite_Subgroups_2010}, Beauville incorrectly insists that the
conic $x^2 + ay^2 + bz^2 = 0$ have a point for \textit{every} $a,b \in G$. For a
counterexample, we take $k = \RR$. Then $\RR^\times / (\RR^\times)^2 = \{\pm
1\}$. The subgroup
  \[
   \Gamma = \left\{\bp  \pm 1 & \\ & 1 \ep, \bp & \pm 1 \\ 1 & \ep \right\}
    \]
    satisfies $\overline{\det}(\Gamma) = \{\pm 1\}$. But the conic $x^2 + y^2 +
    z^2 = 0$ does not have a real point. One can show directly that every
    subgroup of $\PGL_2(\RR)$ isomorphic to $(\ZZ/2\ZZ)^2$ is conjugate to
    $\Gamma$. But since $\RR^\times / (\RR^\times)^2$ has only two subgroups,
    and one of them doesn't satisfy the conic property we require, this fact
    follows immediately from the corrected version of
    Theorem~\ref{thm:beauville_classification}(3).
\end{remark}


\section{Separably Closed Fields}
\label{Sec: Separably closed}

\begin{convention*}
Throughout this section we assume that $k$ is a separably closed field with
algebraic closure~$k_\alg$.
\end{convention*}

The goal of this section is to prove Theorem~C and Theorem~B \textit{in the case
  where $k$ is a separably closed field}.  In this setting, Theorem~B becomes:

\begin{theorem}
\label{Thm: Finite subgroups all p, separably closed}
Let $k$ be a separably closed field of characteristic $p > 0$. 
\begin{enumerate}
  \item Fix $q > 2$, a power of $p$.  There is exactly one conjugacy class of
    subgroups isomorphic to each of $\PSL_2(\FF_q)$ and $\PGL_2(\FF_q)$.

  \item Let $m, n$ be positive integers with $n$ coprime to $p$. The conjugacy
    classes of split $p$-semi-elementary subgroups of order $p^m n$ are
    parameterized by the set of homothety classes of rank-$m$ subgroups $\Gamma
    \subset k$ that are stable under multiplication by elements of
    $\mu_n(k)$. The correspondence is
  \[
  \Gamma \mapsto \bp \mu_n(k) & \Gamma \\ & 1 \ep.
  \]

\item Suppose that $p = 2$, and let $m \ge 1$ be an integer. The conjugacy
  classes of non-split $2$-elementary subgroups of order $2^m$ are parameterized
  by pairs $(k(\tau), G)$, where $k(\tau)$ is a quadratic inseparable extension
  of $k$ and $G$ is a subgroup of order $2^m$ of the abelian group
  \[
 \Omega(\tau) := \{I\} \cup
  \left\{ \bp \alpha & \tau^2 \\ 1 & \alpha \ep \ : \ \alpha \in k \right\}.
  \]      
						
\item Suppose that $p = 2$ and $n > 1$ is an odd integer.  Let
  $\mathfrak{Dih}_n(k)$ denote the set of conjugacy classes of dihedral
  subgroups of $\PGL_2(k)$ of order $2n$. The map $\mathfrak{Dih}_n(k) \to
  k^\times / (k^\times)^2$ defined by $G \mapsto \overline{\det}(t)$ for any
  involution $t \in G$ is well defined and bijective.
					  
\item If $p = 3$, then there is exactly one conjugacy class of subgroups
  isomorphic to $\mathfrak{A}_5$.
\end{enumerate}
Any $p$-irregular subgroup of $\PGL_2(k)$ is among the five types listed here.
\end{theorem}

Tracing through the proofs in Sections~\ref{Sec: Fixed Points}--\ref{Sec:
  p-irregular}, we find that there are a few ways in which the algebraically
closed nature of $k$ was used:
\begin{itemize}

\item To conjugate a stable pair of points to $0$ and $\infty$ (\S\ref{Sec:
  Cyclic}, \S\ref{Sec: Stable pair of points}). If a subgroup $G$ stabilizes a
  pair of distinct points $\{x,y\}$, then any nontrivial $s \in G$ has this
  property. It follows that either $s$ or $s^2$ fixes $x$ and $y$. If $s$ does
  not have order~2, then $x,y$ are distinct roots of a quadratic equation over
  $k$, and they are separable.
  
\item To replace $s$ with $\det(s)^{-1/2}\, s$ in order to assume $\det(s) =
  1$. This was only used in \S\ref{sec:det(s)=1} where $p$ is assumed to be
  odd. In that case, square roots are separable.

\item To assert the existence of a primitive fourth root of unity (\S\ref{Sec:
  Tetrahedral}, \S\ref{Sec: Octahedral}) and a primitive fifth root of unity
  (\S\ref{Sec: Icosahedron}). Roots of unity are always separable. 

\item To find a $k$-rational fixed point of $s = \mat{ \alpha & \beta \\ \gamma
  & \delta}$, which was then conjugated to $\infty$. If $\gamma = 0$, then
  $\infty$ is a fixed point of $s$. So we may suppose $\gamma = 1$. Then the
  equation defining the fixed points of $s$ is $z^2 + (\delta - \alpha) z -
  \beta = 0$. This polynomial is separable --- in which case it will have a root
  in $k$ --- if and only if $p > 2$ or $p =2$ and $\delta \neq \alpha$. So $s$
  fails to have a $k$-rational fixed point if and only if $p = 2$, $\delta =
  \alpha$, and $\beta$ is not a square in~$k$. In particular, $s$ must have
  order~2.

\item To assert the existence of $\sqrt{\tau}$, the fixed point of the element
  $s = \mat{ & \tau \\ 1 & }$ in \S\ref{Sec: Stable pair of points}. This square
  root is separable when $p > 2$, but it may fail to be when $p = 2$. Note that
  $s$ has order~2.
\end{itemize}

This discussion shows that Theorem~C holds for an arbitrary separably closed
field, and Theorem~\ref{Thm: Finite subgroups all p, separably closed} holds
provided the characteristic of $k$ is at least~3.

\begin{convention*}
In the remainder of this section, we assume that $k$ is a separably closed field
of characteristic~2.
\end{convention*}

Let us summarize what we have learned. 

\begin{lemma}
\label{Lem: 2-normal form}
Let $s \in \PGL_2(k)$ be an element with no $k$-rational fixed point.  Then $s =
\mat{ \alpha & \beta \\ 1 & \alpha}$ for some $\alpha, \beta \in k$ with $\beta$
a non-square.  In particular, $s$ has order~2.
\end{lemma}

Theorem~\ref{Thm: Finite subgroups all p, algebraically closed} applied to
$\PGL_2(k_\alg)$ shows that a finite $p$-irregular subgroup of $\PGL_2(k)$ is
either dihedral, isomorphic to $\PGL_2(\FF_q)$ for some $q$, or
$p$-semi-elementary. We treat each of these cases in turn.

\begin{lemma}
  \label{lem:dihedral_well_defined}
  Let $n > 1$ be an odd integer.  Let $\mathfrak{Dih}_{n}(k)$ denote the set of
  $k$-conjugacy classes of dihedral subgroups of $\PGL_2(k)$ of order $2n$.
  Write $D \colon \mathfrak{Dih}_n(k) \to k^\times / (k^\times)^2$ for the map
  that sends $G$ to $\overline{\det}(t)$ for any involution $t \in G$. Then $D$
  is well defined.
\end{lemma}

\begin{proof}
Let $G$ be dihedral of order $2n$, and let $t$ be an involution in
$G$. Replacing $G$ with one of its conjugates has the effect of replacing $t$
with an element of the same determinant. So it suffices to show that every
involution in $G$ has the same determinant, up to squares. Write $G = C \rtimes
\langle t \rangle$ for a cyclic group $C \subset \PGL_2(k)$. Any involution in
$G$ is of the form $st$ for some $s \in C$. Any two cyclic groups of odd order
are conjugate in $\PGL_2(k)$ (Theorem~\ref{thm:beauville_classification}), and
so $usu^{-1} = \mat{ \lambda + 1 & 1 \\ 1 & 1}$ for some $\lambda \in k$ and $u
\in \PGL_2(k)$ (Corollary~\ref{Cor: Alternate Cyclic}). Here $\lambda = \zeta +
\zeta^{-1}$ for some root of unity $\zeta$ with the same order as $s$. Then
$\det(s) = \lambda$. Since $\lambda$ lies in the algebraic closure of $\FF_2$,
it is a square.
\end{proof}

\begin{remark}
  We can simplify the above proof by conjugating $C$ to $\mat{ \mu_n(k) & \\ &
    1}$ since we are working over a separably closed field. However, our
  approach applies over an arbitrary field, and will thus be more useful in the
  final section of these notes.
\end{remark}

\begin{proposition}
\label{Prop: k-dihedral}
Fix an odd integer $n > 1$. Let $\mathfrak{Dih}_{n}(k)$ denote the set of
$k$-conjugacy classes of dihedral subgroups of $\PGL_2(k)$ of order $2n$.  Write
$D \colon \mathfrak{Dih}_n(k) \to k^\times / (k^\times)^2$ for the map that
sends $G$ to $\overline{\det}(t)$ for any involution $t \in G$. Then $D$ is a
bijection.
\end{proposition} 

\begin{proof}
We showed in Lemma~\ref{lem:dihedral_well_defined} that the map $D \colon
\mathfrak{Dih}_n(k) \to k^\times / (k^\times)^2$ given by $D(G) =
\overline{\det}(t)$ for any involution $t$ in $G$ is well defined. To see that
it is surjective, fix $\tau \in k^\times$ and define
\[
G_\tau := \bp\mu_n(k) & \\ & 1 \ep \rtimes \left\langle \bp & \tau \\ 1 &\ep \right\rangle.
\]
Evidently $G_\tau$ is dihedral of order $2n$, and $D(G_\tau) = \tau$. It remains
to prove injectivity.

Let $G$ be a dihedral subgroup of $\PGL_2(k)$ of order $2n$, and let $\zeta \in
k$ be a primitive $n$-th root of unity. We begin by showing that there is $\tau
\in k^\times$ such that $G$ is conjugate to $G_\tau$.  Write $G = C \rtimes
\langle t \rangle$ with $C$ the index~2 cyclic normal subgroup of $G$ and $t$ an
element of order~$2$. Since $C$ has odd order, the quadratic polynomial defining
its fixed points must have distinct --- and hence $k$-rational --- roots. Let us
conjugate them to $0$ and $\infty$ to get a new group $G' = \mat{\mu_n(k) & \\ &
  1} \rtimes \langle t' \rangle$. Then $t'$ acts on the fixed points of
$\mat{\mu_n(k) & \\ & 1}$, but it cannot fix both of them. Hence $t' = \mat{ &
  \tau \\ 1 &}$ for some $\tau \in k^\times$. 

Next we show that  $G_\tau$ and $G_{\tau'}$ lie in the same conjugacy class if
and only if $\tau' = \mu^2 \tau$ for some $\mu \in k^\times$. One direction is
easy: $\mat{ \mu & \\ & 1}G_\tau \mat{ \mu^{-1} & \\ & 1} = G_{\mu^2 \tau}$. For
the other direction, suppose that $s^{-1}G_\tau s = G_{\tau'}$ for some $s \in
\PGL_2(k)$.  Note that $s$ conjugates the index~2 normal subgroup of $G_{\tau'}$
to that of $G_\tau$, which is to say that $s$ lies in the normalizer of
$\mat{\mu_n(k) & \\ & 1}$. Hence $s = \mat{ \mu & \\ & 1}$ or $ \mat{& \mu \\ 1
  &}$ for some $\mu \in k^\times$. In these two cases, we have
\[
s^{-1} \bp & \tau \\ 1 & \ep s = \bp & \mu^{-2} \tau \\ 1 & \ep \quad \text{or} \quad
	\bp & \mu^2 \tau^{-1} \\ 1 & \ep.
\]
Since these elements have order~2 in $G_{\tau'}$, there must be $\zeta \in
\mu_n(k)$ such that
\[
\bp & \tau' \\ 1 & \ep = \bp \zeta & \\ & 1 \ep \bp & \mu^{-2} \tau \\ 1 & \ep 
	\quad \text{or} \quad \bp \zeta & \\ & 1 \ep \bp & \mu^2 \tau^{-1} \\ 1 & \ep.
\]
As the squaring map on $\mu_n(k)$ is surjective, it follows that $\zeta$ is a
square in $k$, and hence $\tau'$ is equal to a square times $\tau$, as desired.
\end{proof}

\begin{lemma}
\label{Lem: small normalizer}
Let $G$ be a finite subgroup of $\PGL_2(k)$ of even order. Define $k(G) / k$ to
be the field extension generated by the fixed points of all Sylow $2$-subgroups
of $G$. (We define $k(\infty) = k$.) Then $k(G) / k$ has degree at most~2.  If
$[k(G) : k] = 2$, then every Sylow $2$-subgroup of $G$ is its own normalizer.
\end{lemma}

\begin{proof}
Let $P$ be a Sylow $2$-subgroup of $G$. Recall that $P$ is abelian and has a
single fixed point in $\PP^1(k_\alg)$, say $\tau$ (Lemma~\ref{Lem: p-group fixed
  point}). Since all Sylow $2$-subgroups are conjugate, we have $k(G) =
k(\tau)$. If $\tau \in \PP^1(k)$, then $k(G) = k$. Otherwise, $P$ does not have
a $k$-rational fixed point. Let $s = \mat{ \alpha & \beta \\ 1 & \alpha}$ be a
nontrivial element of $P$ (Lemma~\ref{Lem: 2-normal form}). Then the fixed point
of $s$ is given by the equation $z^2 = \beta \not\in (k^\times)^2$, so that
$\tau = \sqrt{\beta}$. Therefore $k(G) / k$ has degree~2.
	
Suppose now that $u$ is a nontrivial element of the normalizer of $P$, so that
$u^{-1} P u = P$. It follows that $\tau$ is a fixed point of $u$. Since the
fixed points of $u$ are defined by an equation with $k$-coefficients, and since
$\tau$ is quadratic inseparable, it follows that $\tau$ is the unique fixed
point of $u$. Thus $u$ has order~2 by Lemma~\ref{Lem: 2-normal form}, and $u \in
P$.
\end{proof}

Next, we treat subgroups isomorphic to $\PGL_2(\FF_q) = \PSL_2(\FF_q)$. Note
that when $q = 2$, this group is dihedral and hence has already been dealt with.

\begin{proposition}
Let $q = 2^r$ for some $r > 1$. If $G \subset \PGL_2(k)$ is isomorphic to
$\PGL_2(\FF_q)$, then it is $k$-conjugate to $\PGL_2(\FF_q)$.
\end{proposition}

\begin{proof}
  Let $P$ be a Sylow $2$-subgroup of $G$. As $|\PGL_2(\FF_q)| = q(q^2 - 1)$, the
  order of $P$ is $q$. If the unique fixed point $\tau$ of $P$ is $k$-rational,
  then the argument in \S\ref{Sec: Lambda = squares} shows that $G$ is
  $k$-conjugate to $\PGL_2(\FF_q)$, as desired. Otherwise, $[k(\tau) : k] = 2$,
  and $P$ is its own normalizer (Lemma~\ref{Lem: small normalizer}). But this is
  absurd: the normalizer of a Sylow $p$-subgroup of $\PGL_2(\FF_q)$ is the Borel
  group containing~$P$.
\end{proof}

Finally, we attend to the finite $2$-semi-elementary subgroups of
$\PGL_2(k)$. If $G$ is such a subgroup, then it fixes a unique point $\tau \in
\PP^1(k_\alg)$. If $\tau$ is $k$-rational, then $G$ is split and we may
conjugate $\tau$ to infinity. The discussion in Section~\ref{Sec: Borel} applies
and we obtain the description in Theorem~\ref{Thm: Finite subgroups all p,
  separably closed}. If $\tau$ is not rational, then it must be quadratic
inseparable. In particular, $\tau^2 \in k$.

\begin{lemma}
  Let $\tau \in k_\alg$ be such that $\tau^2 \in k$ but $\tau \not\in
  k$. Suppose that $s \in \PGL_2(k)$ fixes $\tau$. Then
  \[
  s \in \Omega(\tau) =  \{I\} \cup
  \left\{ \bp \alpha & \tau^2 \\ 1 & \alpha \ep \ : \ \alpha \in k
  \smallsetminus \{\tau\} \right\}.
  \]
\end{lemma}

\begin{proof}
  We may suppose that $s$ is nontrivial. Since $\tau$ is inseparable and $s$ is
  defined over $k$, the fixed point polynomial for $s$ vanishes only at
  $\tau$. Apply Proposition~\ref{prop:alternate-elementary}.
\end{proof}

If $G$ is a non-split $2$-semi-elementary subgroup with fixed point $\tau$, then
the lemma shows $G \subset \Omega(\tau)$. If $\tau' \in k(\tau) \smallsetminus
k$, then $\Omega(\tau)$ is conjugate to $\Omega(\tau')$.  Indeed, since $\tau$
is quadratic, we may write $\tau' = u + v\tau$ for some $u,v \in k$ and $v \ne
0$. Looking at fixed points, we see immediately that
\[
  \Omega(\tau)  = \bp v & u \\ & 1 \ep^{-1} \Omega(\tau') \bp v & u \\ & 1 \ep.
  \]
In particular, any finite subgroup of $\Omega(\tau)$ is conjugate to one inside
$\Omega(\tau')$.

It remains to show that if $G, G' \subset \Omega(\tau)$ are subgroups that are
conjugate inside $\PGL_2(k)$, then in fact they are equal. Suppose that
$s^{-1}G's = G$ for some $s \in \PGL_2(k)$ . Then $s$ must fix $\tau$. The above
lemma shows that $s \in \Omega(\tau)$. As $\Omega(\tau)$ is abelian, we conclude
that $G = G'$.


\section{Galois Descent}
\label{Sec: Descent}

\begin{convention*}
For this section, $k$ denotes an arbitrary field of positive characteristic~$p$,
$k_\sep$ denotes a separable closure of $k$, $k_\alg$ denotes an algebraic
closure of $k_\sep$, and $\mathfrak{g} = \mathrm{Gal}(k_\sep / k)$.
\end{convention*}

We use the technique of Galois descent to pass from the classification of finite
subgroups of $\PGL_2$ over separably closed fields to the case of arbitrary
fields. It begins with a result of Beauville:
	
\begin{theorem}[{\cite[\S2]{Beauville_Finite_Subgroups_2010}}]
\label{Thm: Beauville}
Let $G$ be an algebraic group defined over a field $k$, let $H$ be a subgroup of
$G(k)$, let $N = N_{G(k)}(H)$ be its normalizer in $G(k)$, and let $Z$ be the
centralizer of $H$ in $G(k_\sep)$. Write $\Conj(H, G(k) )$ for the set of
subgroups of $G(k)$ that are conjugate to $H$ in $G(k_\sep)$ modulo conjugacy in
$G(k)$. Then there is a canonical isomorphism of pointed sets
\[
\ker\Big[ H^1(\gg, Z) \to H^1\big(\gg, G(k_\sep)\big) \Big] \ / \ N \stackrel{\sim}{\longrightarrow} \Conj(H, G(k)).
\]
\end{theorem}

We are able to give a complete description of the kernel in Beauville's result when $H$ is a $p$-irregular subgroup. 

\begin{theorem}
\label{Thm: H1}
Let $H \subset \PGL_2(k)$ be a finite $p$-irregular subgroup, and let $Z$ be the
centralizer of $H$ in $\PGL_2(k_\sep)$.
\begin{enumerate}
\item If $p > 2$, then $H^1(\gg, Z) = 1$.
\item If $p = 2$ and $H$ is not a 2-group, then $H^1(\gg, Z) = 1$.
\item If $p = 2$ and $H$ is a split 2-group, then $H^1(\gg, Z) = 1$.
\item If $p = 2$ and $H$ is a non-split $2$-group with inseparable fixed point
  $\tau$, then there is an isomorphism of abelian groups
  \[
  k / f(k \times k) \simarrow H^1(\gg,Z),
  \]
  where $f(x,y) = x^2 + y + \tau^2 y^2$.
\end{enumerate}
In all cases, the canonical homomorphism $H^1(\gg,Z) \to H^1(\gg, \PGL_2)$ has
trivial kernel.
\end{theorem}

Applying Theorem~\ref{Thm: Beauville}, we find that $\mathrm{Conj}\left(H,
\PGL_2(k) \right)$ contains a unique element for any finite $p$-irregular
subgroup $H$. We state this formally:

\begin{corollary}
\label{Cor: Descent}
Let $H$ be a finite $p$-irregular subgroup of $\PGL_2(k)$. If $H' \subset
\PGL_2(k)$ is another subgroup that is conjugate to $H$ inside $\PGL_2(k_\sep)$,
then it is already conjugate inside $\PGL_2(k)$.
\end{corollary}

\begin{proof}[Proof of Theorem~\ref{Thm: H1}]
Any element $s \in Z$ must stabilize the unique fixed point of any Sylow
$p$-subgroup of $H$. If $H$ has at least~3 Sylow $p$-subgroups, then $s$ fixes
at least~3 points of $\PP^1(k_\alg)$. That is $Z = 1$, and hence $H^1(\gg, Z) =
1$ for trivial reasons.  Looking at the possible isomorphism types of
$p$-irregular subgroups in Theorem~\ref{Thm: Finite subgroups all p,
  algebraically closed}, we see that this argument covers everything but the
case where $H$ is a $p$-semi-elementary subgroup.

Now suppose that $H$ is a $p$-semi-elementary subgroup of $\PGL_2(k)$ with Sylow
$p$-subgroup $P$. Write $\tau \in \PP^1(k_\alg)$ for the unique fixed point of
$H$. We have two cases to consider depending on whether $\tau$ is $k$-rational. 

\noindent \textbf{Case $\tau$ is $k$-rational.}  We may conjugate $\tau$ to
$\infty$, so that $P \subset \mat{1 & k \\ & 1}$. The centralizer of $P$ in
$\PGL_2(k_\sep)$ is $U(k_\sep) = \mat{1 & k_\sep \\ & 1}$, and hence $Z \subset
U(k_\sep)$. If $P = H$, then $Z = U(k_\sep) \cong \mathbb{G}_a(k_\sep)$ as
$\gg$-modules. It is well known that $H^1(\gg, \mathbb{G}_a) = 1$
\cite[Ch.~X.1]{Serre_Corps_Locaux}. If instead $P \subsetneq H$, then $H$
contains an element $u$ of order prime to $p$. The fixed points of $u$ are
$\infty$ and a second $k$-rational point $\tau'$. Any $s \in Z$ must also fix
$\tau'$, which means $Z = 1$, and $H^1(\gg, Z) = 1$.

Note that the case where $\tau$ is $k$-rational covers the case where $p >
2$. Indeed, the defining equation for the fixed points of a nontrivial element
of $P$ is quadratic with a unique solution; if $p > 2$, then $\tau$ is
automatically $k$-rational.

\noindent \textbf{Case $\tau$ is not $k$-rational.} We must have $p = 2$ and $H$
is a 2-group. Then $\tau$ is purely inseparable over $k$, so $\tau^2 \in k$. The
centralizer of $H$ is the group of elements that fix $\tau$:
\[
Z = \left\{ \bp \alpha & \tau^2 \beta \\ \beta & \alpha \ep \ :  \ 
    \alpha, \beta \in k_\sep, \ \alpha^2 + \tau^2 \beta^2 \neq 0 \right\}.
\]
    
Consider the short exact sequence of $\gg$-module homomorphisms
\[
1 \longrightarrow Z \stackrel{\phi}{\longrightarrow}  \mathbb{G}_a \times \mathbb{G}_a
\stackrel{f}{\longrightarrow} \mathbb{G}_a \longrightarrow 0,
\]
where 
\[
\phi \bp \alpha & \tau^2 \beta \\ \beta & \alpha \ep = \left( \frac{\alpha \beta}{\alpha^2 + \tau^2 \beta^2},
\frac{\beta^2}{\alpha^2 + \tau^2 \beta^2} \right) \quad \text{and} \quad
f(x,y) = x^2 + y + \tau^2 y^2.
\]
(To see that $\ker(f) \subset \im(\phi)$, take $\alpha = x/y$ and $\beta = 1$
when $y \neq 0$. To see that $f$ is onto, take $z \in k_\sep$ and note that
$\tau^2 y^2 + y + z$ is separable as a polynomial in $y$. So $f(0,y) = z$.)
Passing to the long exact sequence on cohomology and using the fact that
$H^1(\gg, \mathbb{G}_a \times \mathbb{G}_a) \cong H^1(\gg, \mathbb{G}_a) \times
H^1(\gg, \mathbb{G}_a) = 1$, we see that
\[
H^0(\gg, \mathbb{G}_a \times \mathbb{G}_a) \stackrel{f}{\longrightarrow}
H^0(\gg, \mathbb{G}_a) \stackrel{\delta}{\longrightarrow} H^1(\gg, Z) \longrightarrow 1.
\]
The coboundary map $\delta$ induces the desired isomorphism $k / f(k\times k)
\simarrow H^1(\gg, Z)$.

Finally, we must show that the kernel of the homomorphism $H^1(\gg,Z) \to
H^1(\gg, \PGL_2)$ is trivial. It is obvious when $H^1(\gg,Z) = 1$, so let us
assume that $p = 2$, $H$ is a 2-group, and its fixed point $\tau$ satisfies
$\sqrt{\tau} \not\in k$.  Let $z_s: \gg \to Z$ be a 1-cocycle that becomes a
coboundary when its target is extended to $\PGL_2$. Then there exists $u \in
\PGL_2(k_\sep)$ such that $z_s = u^{-1} (^su)$ for every $s \in \gg$. Since
every element of $Z$ fixes $\sqrt{\tau}$, we see that
\[
\sqrt{\tau} = z_s.\sqrt{\tau} = u^{-1} (^s u).\sqrt{\tau} \ \Rightarrow \
	u.\sqrt{\tau} = ({^su}).\sqrt{\tau} = {^s(}u.\sqrt{\tau}).
\]
The final equality follows from the fact that $\sqrt{\tau}$ lies in a purely
inseparable extension of $k$, so that $\gg$ acts trivially on it. Thus
$u.\sqrt{\tau}$ is defined over $k(\sqrt{\tau})$. Note that $u.\sqrt{\tau}
\not\in \PP^1(k)$ since that would imply $\sqrt{\tau} = u^{-1}.(u.\sqrt{\tau})
\in \PP^1(k_\sep)$. Choose $v \in \PGL_2(k)$ such that $v.(u.\sqrt{\tau}) =
\sqrt{\tau}$. Then $vu \in Z$, and $z_s = (vu)^{-1} \cdot {^s(vu)}$. That is,
$z_s$ is already a coboundary, and hence trivial in $H^1(\gg,Z)$.
\end{proof}


\section{Proofs of the Main Theorems}
\label{Sec: Proofs}

\begin{convention*}
Throughout this section, $k$ is an arbitrary field of characteristic $p \ge 0$.
\end{convention*}

Theorem C has already been proved in \S\ref{Sec: p-regular} and \S\ref{Sec:
  Separably closed}.  We are now ready to prove Theorems~A and~B from the
introduction.

\begin{proof}[Proof of Theorem~A]
As $\PGL_2(k) \subset \PGL_2(k_\sep)$, the isomorphism type of a finite
$p$-irregular subgroup of $\PGL_2(k)$ must be among those in Theorem~\ref{Thm:
  Finite subgroups all p, separably closed}. 

We begin with subgroups isomorphic to $\PGL_2(\FF_q)$ for some $q$. If $\FF_q
\subset k$, then evidently $\PGL_2(\FF_q) \subset \PGL_2(k)$. Conversely,
suppose that $G \subset \PGL_2(k)$ is isomorphic to $\PGL_2(\FF_q)$. We must
show $\FF_q \subset k$; as this is obvious when $q = 2$, we may assume $q >
2$. Theorem~\ref{Thm: Finite subgroups all p, separably closed} shows that $G$
is conjugate inside $\PGL_2(k_\sep)$ to $\PGL_2(\FF_q)$, and Corollary~\ref{Cor:
  Descent} shows that this conjugacy takes place in $\PGL_2(k)$. It follows that
$\FF_q \subset k$.

An identical argument applies to subgroups isomorphic to $\PSL_2(\FF_q)$.

Next we treat the $p$-semi-elementary subgroups of order $p^m n$ with $m \ge 1$
and $n \ge 1$ coprime to $p$. Let $e$ be the minimum positive integer such that
$p^e \equiv 1 \pmod n$, and suppose that $e \mid m$ and $m \le
\dim_{\FF_p}(k)$. Let $\Gamma \subset k$ be an $\FF_{p^e}$-vector space of
dimension $m / e$. Then $|\Gamma| = p^m$, and $\mat{ \mu_n(k) & \Gamma \\ & 1}$
is a $p$-semi-elementary subgroup of order $p^m n$. Conversely, suppose that $G
\subset \PGL_2(k)$ is a split $p$-semi-elementary group of order $p^m
n$. Writing $e$ for the order of $p$ in $(\ZZ / n\ZZ)^\times$, we must show that
$e \mid m$ and $m \le \dim_{\FF_p}(k)$. Since $G$ is split, its unique fixed
point is $k$-rational. We can conjugate its fixed point to infinity and write $G
= \mat{ \mu_n(k) & \Gamma \\ & 1}$ as in \S\ref{Sec: Borel}. Then $\FF_{p^e}
\subset \FF_\Gamma$ since $\FF_\Gamma$ contains the $n$-th roots of unity. It
follows that $\Gamma$ is an $\FF_{p^e}$-vector space, from which we conclude
that $e \mid m$. As $\Gamma \subset k$, it is clear that $m =
\dim_{\FF_p}(\Gamma) \le \dim_{\FF_p}(k)$.

Now we suppose that $p = 2$ and treat the non-split $2$-elementary subgroups of
order $2^m$, $m \ge 1$. Suppose there is $\tau \in k_\alg$ such that $\tau^2 \in
k$ but $\tau \not\in k$. Evidently $k$ is not finite, and so the group
$\Omega(\tau) \subset \PGL_2(k)$ of Proposition~\ref{prop:alternate-elementary}
is a 2-elementary group of infinite rank. Choose any subgroup of
rank~$m$. Conversely, suppose there exists a non-split $2$-elementary subgroup
$G \subset \PGL_2(k)$ of order $2^m$. By definition, the unique fixed point
$\tau$ of $G$ is not $k$-rational. Let $s \in G$ be non-trivial. By
Proposition~\ref{prop:alternate-elementary}, we see $s = \mat{\alpha & \tau^2
  \\ 1 & \alpha}$ for some $\alpha \in k$. It follows that $\tau^2 \in k$, but
$\tau \not\in k$. That is, $k$ contains a non-square.

Next we assume $p = 2$ and deal with the existence of dihedral subgroups of
order $2n$ with $n > 1$ odd. If $\zeta$ is a primitive $n$-th root of unity such
that $\lambda = \zeta + \zeta^{-1} \in k$, then we set $s = \mat{ \lambda + 1 &
  1 \\ 1 & 1}$ and $t = \mat{ & 1 \\ 1 &}$.  Corollary~\ref{Cor: Alternate
  Cyclic} shows that $s$ has order~$n$, and direct calculation shows that $tst =
s^{-1}$. That is, $\langle s, t \rangle \subset \PGL_2(k)$ is dihedral of the
desired order. Conversely, suppose that $H \subset \PGL_2(k)$ is dihedral of
order~$2n$. In particular, $H$ contains a $2$-regular cyclic subgroup of order
$n$, which shows that $\zeta + \zeta^{-1} \in k$ for some primitive $n$-th root
of unity $\zeta$ (Theorem~\ref{thm:serre_existence}).

Finally, we suppose that $p = 3$ and consider the existence of icosahedral
subgroups of $\PGL_2(k)$. If we assume that $\FF_9 \subset k$, then $k$ contains
a solution $\lambda$ to $X^2 + X + 2 = 0$. Define
\[
s = \bp \lambda + 1 & -1 \\ 1 & 1 \ep \ \quad \ t = \bp -1 & \\ & 1 \ep.
\]
Lemma~\ref{Lem: order n} shows that $s$ has order~5, that $t$ has order~2, and
that $st$ has order~3. It follows that the subgroup $\langle s, t \rangle
\subset \PGL_2(\FF_9)$ is isomorphic to $\mathfrak{A}_5$ (Lemma~\ref{Lem:
  Icosahedral presentation}). Conversely, suppose that $G \subset \PGL_2(k)$ is
an icosahedral subgroup. Then it contains a subgroup of order~5. By
Theorem~\ref{thm:serre_existence}, we see that $ \lambda := \zeta + \zeta^{-1}
\in k$ for some primitive fifth root of unity $\zeta$. But $\lambda$ is
quadratic over $\FF_3$, with minimal polynomial $X^2 + X +2$, and hence it
generates $\FF_9$. 
\end{proof}

\begin{proof}[Proof of Theorem~B]
Over the separable closure $k_\sep$ of $k$, there are five kinds of subgroups to
consider (Theorem~\ref{Thm: Finite subgroups all p, separably closed}). We will
determine the conjugacy classes over $k$ in each case. 

We begin with subgroups isomorphic to $\PGL_2(\FF_q)$ or $\PSL_2(\FF_q)$ with $q
> 2$. Theorem~\ref{Thm: Finite subgroups all p, separably closed} shows there is
a unique $\PGL_2(k_\sep)$-conjugacy class of subgroups isomorphic to either of
these, and Corollary~\ref{Cor: Descent} shows there is at most one conjugacy
class in $\PGL_2(k)$.

The case of split $p$-semi-elementary subgroups of order $p^m n$ is handled just
as in \S\ref{Sec: Borel}. 

The case of non-split $2$-elementary subgroups of order $2^m$ is immediate from
the discussion at the end of \S\ref{Sec: Separably closed}; indeed, the
hypothesis that $k$ was separably closed was never used there. 

Next we assume $p = 2$ and deal with existence of dihedral subgroups of order
$2n$ ($n > 1$ odd). Lemma~\ref{lem:dihedral_well_defined} shows that the map $D
\colon \mathfrak{Dih}_n(k) \to k^\times / (k^\times)^2$ associating $G$ to the
determinant of one of its involutions is well defined. Suppose that $G, G'$ are
dihedral of order $2n$ satisfying $D(G) = D(G')$. Then $G$ and $G'$ are
conjugate over $k_\sep$ by Theorem~\ref{Thm: Finite subgroups all p, separably
  closed}. Corollary~\ref{Cor: Descent} shows that $G$ and $G'$ are conjugate
over $k$. It follows that $D$ is injective. To complete the proof in this case,
we must show that $D$ is surjective when $k$ contains a primitive $n$-th root of
unity, say $\zeta$. But this is clear: given any $\tau \in k^\times$, the group
$G_\tau := \langle \mat{\zeta & \\ & 1}, \mat{& \tau \\ 1&} \rangle$ is dihedral
of order $2n$ and satisfies $D(G_\tau) = \tau$. 

Finally, suppose that $p = 3$, and $\FF_9 \subset k$. Theorem~A shows that
$\PGL_2(k)$ contains an icosahedral subgroup. As above, Theorem~\ref{Thm: Finite
  subgroups all p, separably closed} and Corollary~\ref{Cor: Descent} show that
there is at most one conjugacy class of such subgroups in $\PGL_2(k)$.
\end{proof}
	
\begin{proof}[Proof of Theorem~D]
Let $\FF_q$ be a finite field with $q = p^r$ elements. Theorems~A and~C give the
list of possible isomorphism types of subgroups for us to consider: cyclic,
dihedral, tetrahedral, octahedral, icosahedral, split $p$-semi-elementary, and
$\PSL/\PGL$. We look at each in turn, though there will be overlap among the
cases.

Suppose first that $C = \langle s \rangle$ is a cyclic subgroup of
$\PGL_2(\FF_q)$ of order $n > 2$ and coprime to $p$.
Theorem~\ref{thm:serre_existence} asserts that $\lambda := \zeta + \zeta^{-1}
\in \FF_q$ for some primitive $n$-th root of unity $\zeta$. If $\zeta \in
\FF_q$, then $n \mid (q-1)$. Otherwise, $\zeta$ is quadratic over $\FF_q$ and $n
\mid (q+1)$. That is, $q \equiv \pm 1 \pmod n$. Conversely, if $\lambda := \zeta
+ \zeta^{-1} \in \FF_q$ for some primitive $n$-th root of unity $\zeta$, then
$\mat{ \lambda + 1 & -1 \\ 1 & 1} \in \PGL_2(k)$ has order $n$
(Corollary~\ref{Cor: Alternate Cyclic}).
Theorem~\ref{thm:beauville_classification} shows that if a cyclic subgroup of
order $n$ exists, then it is unique up to conjugation.

Next we look at cyclic subgroups of order~2 when $p \ne 2$. The element $\mat{&1
  \\ 1&}$ has order~2, so such subgroups always exist. Since $\FF_q^\times /
(\FF_q^\times)^2$ has order~2, Theorem~\ref{thm:beauville_classification}(3)
shows that there are two distinct conjugacy classes. They may be represented by
$\{I, \mat{&1 \\ 1&}\}$ and $\{I, \mat{ & \zeta \\ 1 &}\}$ for $\zeta$ a
primitive $(q-1)$-st root of unity. The former is split and the latter is not.

Now we turn to the dihedral groups, beginning with those of order $2n$ where $n
\ge 3$ and $p$ is coprime to $2n$. Suppose first that $q \equiv 1 \pmod{n}$, so
that $\FF_q$ contains a primitive $n$-th root of
unity~$\zeta$. Theorem~~\ref{thm:beauville_classification}(5) asserts that
conjugacy classes of dihedral subgroups of order $2n$ are in bijection with
elements of the group $\FF_q^\times / \mu_n(\FF_q)(\FF_q^\times)^2$. If $q
\equiv 1 \pmod{2n}$, then every $n$-th root of unity is a square and this group
has order $2$; otherwise, it has order~1.


Next we consider dihedral groups of order $2n$ where $q \equiv -1 \pmod{n}$ and
$n \ge 3$. First note that such subgroups always exist: if $\zeta \in
\FF_{q^2}^\times$ is a primitive $n$-th root of unity, then $\lambda := \zeta +
\zeta^{-1} \in \FF_q$ and $\left \langle \mat{ \lambda+1 & -1 \\ 1 & 1}, \mat{
  -1 & \lambda \\ & 1}\right\rangle$ is dihedral of the desired order
(cf. Corollary~\ref{Cor: Alternate Cyclic}). Suppose now that $H \subset
\PGL_2(\FF_q)$ is any dihedral subgroup of order~$2n$. Let $s$ be a generator of
the cyclic subgroup of index~2. Since cyclic $p$-regular subgroups are unique up
to conjugation, we may assume that $s = \mat{ \lambda+1 & -1 \\ 1 & 1}$. Let $t$
be an involution in $H$, so that $H = \langle s,t\rangle$. Set $u = \mat{1 &
  -\zeta^{-1} \\ 1 & -\zeta} \in \PGL_2(\FF_{q^2})$; conjugating by $u$ has the
effect of moving the fixed points of $s$ to $\infty$ and $0$:
\[
u s u^{-1} = \bp \zeta & \\ & 1 \ep.
\]
Since $s$ fixes $\zeta$ and $\zeta^{-1}$, $t$ must swap these points. It follows
that there is $\tau \in \FF_{q^2}^\times$ such that
\[
utu^{-1} = \bp & \tau \\ 1 &\ep.
\]
We can recover $t$ from $\tau$:
\[
u^{-1}\bp &\tau \\ 1& \ep u =
\bp \zeta^{-1} - \tau \zeta & -\zeta^{-2} + \tau\zeta^2 \\
1 - \tau & -\zeta^{-1} + \tau \zeta \ep.
\]
If $\tau = 1$, then $u^{-1} \mat{ & \tau \\ 1 &} u = \mat{ -1 & \lambda \\ &
  1} \in \PGL_2(\FF_q)$. Otherwise, we find that
\[
u^{-1}\bp &\tau \\ 1& \ep u =
\bp \frac{\zeta^{-1} - \tau \zeta}{1-\tau} & \frac{-\zeta^{-2} + \tau\zeta^2}{1-\tau} \\ 1 & \frac{-\zeta^{-1} + \tau \zeta}{1-\tau} \ep \qquad (\tau \ne 1).
\]
Since this must lie in $\PGL_2(\FF_q)$, the top left entry must lie in
$\FF_q$. That is, it is Frobenius invariant, and we have
\[
\frac{\zeta - \tau^q \zeta^{-1}}{1-\tau^q} = \frac{\zeta^{-1} - \tau \zeta}{1-\tau}.
\]
Clearing denominators and massaging shows that $\tau^{q+1} = 1$. Thus $\tau \in
(\FF_{q^2}^\times)^{q-1}$.

We wish to determine how $\tau$ depends on the conjugacy class of $H$. The
normalizer of $uHu^{-1}$ in $\PGL_2(\FF_{q^2})$ consists of elements of the form
$\mat{ \alpha & \\ & 1}$ and $\mat{ & \beta \\ 1 &}$, with $\alpha,\beta \in
\FF_{q^2}^\times$. We restrict attention to elements $v$ of this form with
$u^{-1}vu \in \PGL_2(\FF_q)$. A computation as in the previous paragraph shows
that $\alpha^{q+1} = \beta^{q+1} = 1$. Conjugating $uHu^{-1}$ by $\mat{\alpha &
  \\ & 1}$ has the effect of replacing $\tau$ with $\alpha^2 \tau$. Conjugating
by $\mat{& \beta \\ 1 &}$ replaces $\tau$ with $\beta^2 \tau^{-1} =
(\beta/\tau)^2 \tau$. Note that $uHu^{-1}$ also contains the involutions
$\zeta^i \tau$ for any $i$. It follows that for any $\tau' \in
\mu_n(\FF_{q^2}^\times) (\FF_{q^2}^\times)^{2(q-1)}$, the group $H$ is conjugate
in $\PGL_2(\FF_q)$ to the group $H'$, where $uH'u^{-1} = \left \langle
\mat{\zeta & \\ & 1}, \mat{ & \tau' \\ 1 &} \right\rangle$. Thus, the conjugacy
classes of dihedral subgroups of $\PGL_2(\FF_q)$ of order $2n$ are in bijection
with elements of the group $(\FF_{q^2}^\times)^{q-1} / \mu_n(\FF_{q^2})
(\FF_{q^2}^\times)^{2(q-1)}$. (Compare
Theorem~\ref{thm:beauville_classification}(5).)  If $q \equiv -1 \pmod{2n}$,
then every element of $\mu_n(\FF_{q^2})$ is a square in $\mu_{q+1}(\FF_{q^2})$
and $(\FF_{q^2}^\times)^{q-1} / \mu_n(\FF_{q^2}) (\FF_{q^2}^\times)^{2(q-1)}$
has order~2. Otherwise, it has order~1.

Now we look at 4-groups with $p \ne 2$. Since every conic over $\FF_q$ has a
point, and since $\FF_q^\times / (\FF_q^\times)^2$ has order~2,
Theorem~\ref{thm:beauville_classification}(4) shows that there are exactly two
conjugacy classes of subgroups of order~4.

Next we consider dihedral subgroups of order $2n$ when $p = 2$ and $n > 1$ is
odd. We may suppose that $q \equiv \pm 1 \pmod{n}$, so that cyclic subgroups of
order~$n$ exist. Since $\FF_q^\times = (\FF_q^\times)^2$, Theorem~B(4) shows
that there is a unique conjugacy class of dihedral subgroups of order $2n$.

Next we consider subgroups isomorphic to $\mathfrak{A}_4$. If $p \ne 2$, then
Theorem~\ref{thm:serre_existence} shows that such a subgroup exists if and only
if $-1$ is the sum of two squares in $\FF_q$, which always holds by a counting
argument. If $p = 2$, then Theorem~\ref{thm:serre_existence} shows that such a
subgroup exists if and only if $\FF_4 \subset \FF_q = \FF_{2^r}$, which is to
say $r$ is even.  Theorem~\ref{thm:beauville_classification} shows that such a
subgroup is unique up to conjugation when $p \ne 2,3$. When $p = 3$, we find
that $\mathfrak{A}_4 \cong \PSL_2(\FF_3)$ (Corollary~\ref{cor:
  tetrahedral}). Theorem~B(1) shows that such a subgroup is unique up to
conjugation. When $p = 2$, a tetrahedral subgroup is $2$-semi-elementary; the
argument used to prove Proposition~\ref{Prop: Tetrahedral 2} shows that it is
conjugate to the standard Borel subgroup in $\PGL_2(\FF_4)$.

The case of octahedral subgroups is identical to that of tetrahedral subgroups,
except that they cannot occur when $p = 2$. Indeed, $\mathfrak{S}_4$ contains an
element of order~4, while $\PGL_2(\FF_{2^n})$ does not. Use
Corollary~\ref{cor:octahedral3} to handle the case $p=3$.

Next we look at subgroups isomorphic to $\mathfrak{A}_5$. If $p \ne 2, 5$, then
Theorem~\ref{thm:serre_existence} shows that such a subgroup exists if and only
if $-1$ is the sum of two squares in $\FF_q$ (always true) and $\sqrt{5} \in
\FF_q$. By quadratic reciprocity, the latter condition holds if and only if $q
\equiv \pm 1 \pmod{5}$. The case $p = 5$ is taken care of by the fact that
$\PGL_2(\FF_5) \cong \mathfrak{A}_5$ (Proposition~\ref{prop:icosahedral5}). If $p = 2$, then
Theorem~\ref{thm:serre_existence} shows that such a subgroup exists if and only
if $-1$ is the sum of two squares in $\FF_q$ (always true) and $\FF_4 \subset
\FF_q = \FF_{2^r}$, which is to say $r$ is even. Note that $r$ is even if and
only if $q \equiv \pm 1 \pmod{5}$. Theorem~\ref{thm:beauville_classification}
shows that there is at most one conjugacy class of icosahedral subgroups when $p
\ne 2, 3, 5$. Theorem~B(1,5) yield the same conclusion for $p = 3, 5$. When $p =
2$, we note that $\PGL_2(\FF_4) \cong \mathfrak{A}_5$
(Theorem~\ref{thm:serre_existence}(4)); Theorem~B(1) shows that the conjugacy
class of icosahedral subgroups is unique in this case.

Next, let us look at the subgroups isomorphic to $\PSL_2(\FF_{p^s})$ and
$\PGL_2(\FF_{p^s})$. Evidently, such subgroups exist if $s \mid r$, since then
$\FF_{p^s} \subset \FF_q$. Theorem~A(1) shows this is a sufficient condition as
well, and Theorem~B(1) shows that there is a unique conjugacy class of such
subgroups.

Finally, we treat the $p$-semi-elementary subgroups of $\PGL_2(\FF_q)$. Every
such subgroup has a unique fixed point, which must be $\FF_q$-rational since
this field is perfect. Let $G$ be such a subgroup, and suppose that its order is
$p^m n$ with $n$ coprime to $p$. Theorem~B(2) shows that $G$ is conjugate to
$\mat{ \mu_n(\FF_q) & \Gamma \\ & 1}$ for some rank-$m$ subgroup $\Gamma \subset
\FF_q$ that is stable under multiplication by elements of $\mu_n(\FF_q)$. In
particular, note that $m \le r$. Let $e$ be the order of $p$ in $(\ZZ /
n\ZZ)^\times$. Then $n \mid (p^e - 1)$, which implies $\FF_{p^e} \subset
\FF_q$. That is, $e \mid r$. Since $\Gamma$ is stable under the multiplicative
action of $\FF_{p^e}$, it follows that it is an $\FF_{p^e}$-vector space. Thus,
$e \mid m$.  Conversely, let us suppose that $e$ is the order of $p$ in $(\ZZ /
n\ZZ)^\times$, that $e \mid\gcd(r,m)$, and that $m \le r$. Then $\FF_q$ contains
a primitive $n$-th root of unity, and $\mat{\mu_n(\FF_q) & \FF_{p^m} \\ &
  1}$ is a $p$-semi-elementary subgroup of the desired order.  Theorem~B(2)
shows that the conjugacy classes of such subgroups are given by homothety
classes $\FF_{p^e}$-vector subspaces of $\FF_q$.
\end{proof}


\bibliographystyle{plain}
\bibliography{finite_subgroups}

\end{document}